\documentclass[preprint,12pt]{elsarticle}

\usepackage{graphicx}          % Include this line if your 
\usepackage[utf8]{inputenc}
\usepackage{comment}
\usepackage{url}
\usepackage{amscd}
\usepackage{amsmath}
\usepackage{eucal}
\usepackage{amssymb}
\usepackage{pstricks}
\usepackage{graphicx}
\usepackage{lineno}
\usepackage{amsthm}

\usepackage{tikz-cd}
\usetikzlibrary{matrix}
\usepackage{tikz}

\newtheorem{definition}{Definition}
\newtheorem{theorem}{Theorem}

\newtheorem{lemma}{Lemma}

\newtheorem{proposition}{Proposition}
\newtheorem{remark}{Remark}

\usepackage{empheq}

\numberwithin{equation}{section}

\begin{document}

\begin{frontmatter}

\title{Reciprocal relationship between detectability and observability in a non-uniform setting} % Title, preferably not more 

 % \thanks[footnote]{The first author was partially funded by FONDECYT Regular Grant 1240361.}

\author[Huerta]{Ignacio Huerta}\ead{ignacio.huertan@usm.cl}    % Add the 
\author[Monzón]{Pablo Monzón}\ead{monzon@fing.edu.uy}               % e-mail address 
% \author[Baiae]{Publius Maro Vergilius}\ead{vergilius@culture.ir}  % (ead) as shown

\address[Huerta]{Departamento de Matem\'atica, Universidad T\'ecnica Federico Santa Mar\'ia, Casilla 110-V, Valpara\'iso, Chile.}  % Please supply                                              
\address[Monzón]{Facultad de Ingenier\'ia, Universidad de la Rep\'ublica, Julio Herrera y Reissig 565, Montevideo, Uruguay.}             % full addresses
% \address[Baiae]{The White House, Baiae}        % here.

                             % keyword list or with the 
                                          % help of the Automatica 
                                          % keyword wizard

\begin{abstract}                          Building on the recent notion of non-uniform complete observability, and on the fact that this property ensures non-uniform exponential detectability, this paper establishes the converse implication under suitable additional assumptions. Specifically, we investigate conditions under which non-uniform exponential detectability guarantees non-uniform complete observability. Our approach is based on a refined analysis of the associated output feedback systems and on the preservation of non-uniform observability properties under such feedback transformations. These results extend the classical equivalence between observability and detectability beyond the uniform framework and provide new tools for the qualitative analysis of time-varying control systems.

\end{abstract}

\end{frontmatter}

\section{Introduction}

Observability of a system's state is a key issue in the analysis and control of dynamical systems, essentially referring to the capability of deducing the internal state from the system's output. Rudolph Kalman introduced the foundational concepts in the 1960s, creating an entire algebraic framework for studying this property. Since then, various studies have addressed related challenges, such as observer design, classes of observability, and its strong connection to controllability, which is the ability to guide a system to a desired state. Kalman initially defined observability and controllability, focusing quickly on linear time-invariant systems, where both properties have shared intrinsic algebraic characteristics and can be linked through dual and adjoint systems. He showed that a similar relationship exists in linear uniform time-variant systems using Gramian matrices. Controllability is related to the weaker notion of stabilizability, i.e., the ability to asymptotically drive the state  to the origin through a state feedback, and observability is closely related to detectability, the existence of an asymptotical observer \cite{Kwakernaak}. Recently, the concepts of non-uniform controllability and observability emerged, leading to a new category of linear time-varying systems that lie between the uniform and general cases \cite{HM,HMR}.

Detectability and stabilizability play a central role in the qualitative ana\-ly\-sis of dynamical systems, particularly when observability or controllability fail to hold globally. Detectability, in its classical formulation, requires that all unobservable components of a system be asymptotically stable, ensuring that the internal state can still be inferred with arbitrary accuracy over time despite incomplete information \cite{Kailath}. In the setting of linear time-invariant systems, Kalman showed that detectability admits a clean algebraic cha\-rac\-te\-ri\-za\-tion, and it is deeply connected to the existence of asymptotic observers and stabilizing solutions of algebraic Riccati equations \cite{Khalil,Sontag, Terrell}. This notion extends naturally—though in a more subtle manner—to linear time-varying systems, where the observability Gramian may fail to provide a uniform quantitative measure of state distinguishability over time, and the asymptotic behavior is inherently governed by the nonautonomous nature of the dynamics. In such cases, uniform detectability guarantees that all non-observable components decay at an exponential rate governed by a uniform Lyapunov function or a uniform Riccati-type inequality \cite{Tranninger}. However, the recently developed non-uniform framework reveals that detectability can persist even when exponential decay varies with time or when growth conditions are not uniformly bounded \cite{HM}. This generalized detectability captures systems where decay rates may depend on the temporal index but remain sufficiently strong to ensure asymptotic reconstruction of the internal state.
\\
Stabilizability exhibits a fully dual structure: it requires that all uncontrollable components be exponentially stable, allowing the design of a feedback control that stabilizes the overall system. In linear time-invariant and uniformly time-varying systems, stabilizability is equivalent to the existence of a feedback gain ensuring uniform exponential stability of the closed-loop system, often characterized using the controllability Gramian or algebraic/differential Riccati equations \cite{Anderson, Anderson69}. The non-uniform theory exposes richer geometric and dynamical behaviors, where stabilizability may hold even if the underlying system lacks uniform bounded growth or fails to satisfy uniform complete controllability. Here, non-uniform stabilizability means that uncontrollable subspaces still decay at time-dependent exponential rates, enabling the construction of nonautonomous feedback gains adapted to the varying dynamics. Moreover, in this generalized setting, detectability and stabilizability retain their deep interconnection through duality and Riccati-type methods provide a unifying framework for both properties under non-uniform assumptions. This dual perspective is crucial when studying output-feedback mechanisms, since maintaining observability via output injection in the non-uniform regime ultimately depends on en\-su\-ring detectability of the augmented system while simultaneously preserving stabilizability of its dual counterpart.

In this work we analyze  how observability is preserved through output feedback in the nonuniform setting. Based on the above, we also prove under what conditions nonuniform exponential detectability guarantees nonuniform complete observability.

\subsection{Structure of the article and notations}
The remainder of the article is organized as follows. The Section 2 consists of all the preliminaries necessary for understanding the concepts of con\-tro\-lla\-bi\-li\-ty and observability at their different levels; complete and uniform, in addition to their relationship with the well-known Kalman condition and its dual connection. After that, Section 3 presents the definitions and results of con\-tro\-lla\-bi\-li\-ty, observability, and Kalman condition in the non-uniform framework, highlighting the fact that the non-uniform properties of con\-tro\-lla\-bi\-li\-ty and observability are halfway between their respective complete and uniformly complete notions. Section 4 presents the first important result of this work: observability is preserved via output feedback in the non-uniform context. Subsequently, Section 5 presents and proves the main result of our article: non-uniform exponential detectability and other additional conditions allow us to conclude non-uniform complete observability, closing the pseudo-equivalence between these two concepts. Based on the duality of the concepts of controllability and observability, Section 6 presents results similar to those proved in the context of observability; controllability is preserved via state feedback, and subsequently, non-uniform exponential stabilizability implies, under additional hypotheses, the non-uniform complete controllability property.

We recall the following notations which will be used in this paper. 
A symmetric matrix $R=R^{T}\in M_{n}(\mathbb{R})$ is semi--positive definite if for any real vector $x\neq 0$, we have that
$x^{T}Rx =\langle R x, x\rangle  \geq 0$, and this property will be denoted as $M \geq 0$. In case that the above inequalities are strict, we say that $R=R^{T}$ is positive definite. Given two matrices $R, S\in M_{n}(\mathbb{R})$, we write $R\leq S$ if  $R-S\geq0$.

\section{Preliminaries} 

Consider the nonautonomous linear control system
\begin{subequations}\label{control}
	\begin{empheq}[left=\empheqlbrace]{align}
		& \dot{x}(t)=A(t)x(t)+B(t)u(t), \label{control1a} \\
		& y(t)=C(t)x(t), \label{control1b}
	\end{empheq}
\end{subequations}
where $t\to x(t)\in\mathbb{R}^{n}$ is the \textit{state} of the system, $t\to u(t)\in\mathbb{R}^{p}$ is the \textit{input}, $t\to y(t)\in\mathbb{R}^{m}$ is the \textit{output}, while $t\mapsto A(t)\in M_{n\times n}(\mathbb{R})$, $t\mapsto B(t)\in M_{n\times p}(\mathbb{R})$ and $t\mapsto C(t)\in M_{m\times n}(\mathbb{R})$ are matrix valued functions for any $t\in\mathbb{R}$. In addition, when considering the null input  $u(t) \equiv 0$, this enables us to call the linear system 
\begin{equation}
\label{lin}
\dot{x}=A(t)x
\end{equation}
as the \textit{plant} or
the \textit{uncontrolled part} of the control system \eqref{control1a}. The transition matrix of \eqref{lin} is denoted by $\Phi_{A}(t,s)$. Now, any solution of \eqref{lin} passing through
$x_{0}$ at $t=t_{0}$, is denoted by $t\mapsto x(t,t_{0},x_{0})=\Phi_{A}(t,t_{0})x_{0}$. In addition, given an input $u$, any solution of \eqref{control1a} passing through the initial condition $x_{0}$ at time $t=t_{0}$ will be denoted by $t\mapsto x(t,t_{0},x_{0},u)$. The solution of \eqref{control1a} is given by the expression:
\begin{equation}\label{eq:solution}
x(t,t_{0},x_0,u)=\Phi_{A}(t,t_{0})x_0+\int_{t_0}^t\Phi_{A}(t,\tau)B(\tau)u(\tau)d\tau.   
\end{equation}

The properties of \textit{controllability} and \textit{observability} were introduced by R. Kalman in \cite{Kalman} for the linear control system \eqref{control}\footnote{\eqref{control} refers to the control system \eqref{control1a}-\eqref{control1b}.} and we will briefly recall them by following \cite{Anderson,Sontag}:
\begin{definition}
% \label{DC1}
The state $x_{0}\in \mathbb{R}^{n}$ of the control system \eqref{control1a} is con\-tro\-lla\-ble at time $t_{0}\in\mathbb{R}$, if there exists an input $u\colon [t_{0},t_{f}]\to \mathbb{R}^{p}$ such that $x(t_{f},t_{0},x_{0},u)=0$. In addition, the control system \eqref{control1a} is: 
\begin{itemize}
\item[a)] \textbf{Controllable at time $t_{0}\in\mathbb{R}$} if any state $x_{0}$ is controllable at time $t_{0}\in\mathbb{R}$,
\item[b)] \textbf{Completely controllable} \textnormal{(CC)} if it is controllable at any time $t_{0}\in\mathbb{R}$.
\end{itemize}
\end{definition}
 
There exists a well known necessary and sufficient condition ensuring both controllability at time $t_{0}$ and complete controllability, which is stated in terms of the controllability Gramian matrix, usually defined by
\begin{equation*}
% \label{contmatrix}
W(a,b)=\displaystyle\int_{a}^{b}\Phi_{A}(a,s)B(s)B^{T}(s)\Phi_{A}^{T}(a,s)\;ds.
\end{equation*}

The control system \eqref{control1a} is {\it controllable at time $t_{0}\in\mathbb{R}$} if and only if there exists
$t_{f}>t_{0}$ such that $W(t_{0},t_{f})>0$. In addition, the system \eqref{control1a} is {\it completely controllable} if and only if for 
any $t_{0}\in\mathbb{R}$, there exists $t_{f}>t_{0}$ such that $W(t_{0},t_{f})>0$. We refer the reader to \cite{Anderson67,Kreindler} for a detailed description. 
\begin{comment}
Essentially, invertibility of the Gramian allows us to construct an explicit input function that drives the system towards the origin. By considering  a given couple of initial state $x_0$ and initial time $t_0$, together with the following input $u^{*}\colon [t_{0},t_{f}]\to \mathbb{R}^{p}$ described by:
\begin{equation}
\label{input}
u^{*}(t)=-B^T(t)\Phi_{A}^T(t_0,t)W^{-1}(t_0,t_f)x_0,
\end{equation}
we can see by (\ref{eq:solution}) that the respective solution of \eqref{control1a} with $u=u^{*}$ satisfies: 
\begin{displaymath}
\begin{array}{rl}
x(t_f,t_{0},x_{0},u^{*})&=\displaystyle\Phi_{A}(t_f,t_0)x_0-\int_{t_0}^{t_f}\Phi_{A}(t_f,\tau)B(\tau)B^T(\tau)\Phi_{A}^T(t_0,\tau)W^{-1}(t_0,t_f)x_0d\tau,\\\\
&=\displaystyle\Phi_{A}(t_f,t_0)x_0-\Phi_{A}(t_f,t_0)W(t_0,t_f) W^{-1}(t_0,t_f)x_0=0.
\end{array}
\end{displaymath}
\end{comment}

It is important to emphasize that, in some references, the controllability Gramian matrix is also denoted by 
\begin{equation*}
% \label{Kcontmatrix}
K(a,b)=\displaystyle\int_{a}^{b}\Phi_{A}(b,s)B(s)B^{T}(s)\Phi_{A}^{T}(b,s)\;ds,
\end{equation*}
which verifies the following properties
\begin{equation}
\label{GP}
\left\{\begin{array}{rcl}
K(a,b)&=&\Phi_{A}(b,a)W(a,b)\Phi_{A}^{T}(b,a) \\
W(a,b)&=&\Phi_{A}(a,b)K(a,b)\Phi_{A}^{T}(a,b),
\end{array}\right.
\end{equation}
and the above mentioned controllability condition can also be stated in terms of invertibility of $K(t_{0},t_{f})$.

A stronger property of controllability, also due to R. Kalman \cite{Kalman,Kalman69}, is given by the \textit{uniform complete controllability} (UCC). In this case, there exists a time $\sigma>0$ such that in the previous Gramian matrices, we have $t_{f}=t_{0}+\sigma$. The following definition formalizes this idea.

\begin{definition}    The linear control system \eqref{control1a} is said to be uniformly completely controllable on $\mathbb{R}$ if there exist a fixed constant $\sigma>0$ and positive numbers $\alpha_0(\sigma)$, $\beta_{0}(\sigma)$, $\alpha_1(\sigma)$ and $\beta_1(\sigma)$ such that the following relations hold for all $t\in\mathbb{R}$: 
\begin{equation}
\label{UC1}
0<\alpha_0(\sigma)I\leq W(t,t+\sigma)\leq \alpha_1(\sigma)I,
\end{equation}
\begin{equation}
\label{UC2}
0<\beta_0(\sigma)I\leq K(t,t+\sigma) \leq \beta_1(\sigma)I.
\end{equation}
\end{definition}
\begin{comment}
By using (\ref{GP}), we can see that the condition (\ref{UC2}) is equivalent to
\begin{displaymath}
0<\beta_0(\sigma)I\leq  \Phi_{A}(t+\sigma,t)W(t,t+\sigma)\Phi_{A}^{T}(t+\sigma,t)  \leq \beta_1(\sigma)I,
\end{displaymath}
which is widely employed in the literature. Nevertheless, by following \cite[p.114]{ZA}, we adopted (\ref{UC2}) by its practical convenience. It is important to emphasize that, in several references, the explicit dependence of $\alpha_{i}$ and $\beta_{i}$ ($i=0,1$)
with respect to $\sigma$ is not considered
\end{comment}

\begin{definition}
% \label{DC1}
The state $x_{0}\in \mathbb{R}^{n}$ of the control system \eqref{control} is observable in the interval $[t_0,t_{f}]$ if given an input $u\colon [t_{0},t_{f}]\to \mathbb{R}^{p}$ and the respective output $y:[t_0,t_{f}]\to\mathbb{R}^{m}$, $y(t)=C(t)x(t,t_0,x_0,u)$, the initial condition $x_0$ can be uniquely estimated from $y(t)$.

In addition, the control system \eqref{control} is: 
\begin{itemize}
\item[a)] \textbf{Observable in the interval $[t_0,t_{f}]$} if any state $x_{0}$ is observable in the interval $[t_0,t_{f}]$.
\item[b)] \textbf{Completely observable} \textnormal{(CO)} if it is observable at any interval $[t_0,t_{f}]$.
\end{itemize}
\end{definition}

\begin{remark}
 Taking into account \eqref{eq:solution}, the previous definition can be understood as the injectivity of the linear map $x_0\to C(\cdot)\Phi_{A}(\cdot,t_0)x_0|_{[t_0,t_f]}$. In other words, the system \eqref{control} is observable in the interval $[t_0,t_{f}]$ if for a given input and for any two different initial conditions, the respective outputs are distinct (see \textnormal{\cite{AEYELS1982,Callier}} for more details), noting that the matrix $B(t)$ does not influence this definition. 
\end{remark}

Similarly to the case of controllability, the observability Gramian matrix is defined as follows:
\begin{equation}
	\label{gramianobservability}
	M(a,b)=\displaystyle\int_{a}^{b}\Phi_{A}^{T}(s,a)C^{T}(s)C(s)\Phi_{A}(s,a)\;ds,
\end{equation}
The control system \eqref{control} is \textit{observable} in the interval $[t_0,t_{f}]$ if and only if $M(t_0,t_{f})>0$. In addition, the system \eqref{control} is \textit{completely observable} if for any $t_0\in\mathbb{R}$, there exists $t_{f}>t_0$ such that $M(t_0,t_{f})>0$. 

We note that there are references (for example \cite[Definition 2.6]{Ikeda2}) where another observability Gramian matrix is defined:
\begin{equation*}
% \label{Kcontmatrix}
N(a,b)=\displaystyle\int_{a}^{b}\Phi_{A}^{T}(s,b)C^{T}(s)C(s)\Phi_{A}(s,b)\;ds,
\end{equation*}
which verifies the following properties
\begin{equation}
\label{GPObservabilidad}
\left\{\begin{array}{rcl}
N(a,b)&=&\Phi_{A}^{T}(a,b)M(a,b)\Phi_{A}(a,b) \\
M(a,b)&=&\Phi_{A}^{T}(b,a)N(a,b)\Phi_{A}(b,a),
\end{array}\right.
\end{equation}
therefore, we can describe a definition in the uniform framework of this concept. 

\begin{definition}\label{UCO}    The linear time-varying control system \eqref{control} is said to be uniformly completely observable \textnormal{(UCO)} on $\mathbb{R}$ if there exists a fixed constant $\sigma>0$ and positive numbers $\bar{\alpha}_0(\sigma)$, $\bar{\beta}_0(\sigma)$, $\bar{\alpha}_1(\sigma)$ and $\bar{\beta}_1(\sigma)$ such that the following relations hold for all $t\in\mathbb{R}$: 
\begin{equation}
\label{UCO1}
0<\bar{\alpha}_0(\sigma)I\leq M(t,t+\sigma)\leq \bar{\alpha}_1(\sigma)I,
\end{equation}
\begin{equation}
\label{UCO2}
0<\bar{\beta}_0(\sigma)I\leq N(t,t+\sigma)\leq \bar{\beta}_1(\sigma)I.
\end{equation}
\end{definition}

\begin{remark}
    Throughout this work, it will be useful to consider the alternative notation $M_{A}(\cdot,\cdot)$ and $N_{A}(\cdot,\cdot)$ instead of $M(\cdot,\cdot)$ and $N(\cdot,\cdot)$ respectively, referring to the dependence of matrix $A(t)$ in the definition of both Gramian matrices.
\end{remark}

Based on the definitions of uniform complete controllability and uniform complete observability, Kalman \cite{Kalman} included a property related to the behavior of the transition matrix of system \eqref{lin}, which is known as the \textit{Kalman condition}, and states that there exists a function $\alpha\colon [0,+\infty)\to (0,+\infty)$, with $\alpha(\cdot)\in\mathcal{B}$\footnote{$\mathcal{B}$ denotes the set of functions $\alpha\colon \mathbb{R}\to \mathbb{R}$ mapping bounded sets into bounded sets.}, such that \eqref{lin} has a transition matrix verifying the property
\begin{equation}
\label{BG}
\|\Phi_{A}(t,s)\| \leq \alpha(|t-s|) \quad \textnormal{for any $t,s \in\mathbb{R}$}.
\end{equation}

In addition, in \cite[Lemma 6]{HMR} it was proved that the Kalman condition is equivalent to the \textit{uniform bounded growth} property, which states that there exist constants $K>1$ and $\beta>0$ such that:
\begin{equation}
\label{BG3}
\|\Phi_{A}(t,s)\|\leq Ke^{\beta|t-s|}\quad \textnormal{for any $t,s\in \mathbb{R}$}.
\end{equation} 

The next result, whose proof has been sketched by Kalman in \cite[p.157]{Kalman}, des\-cribes a relation between the properties of uniform complete con\-tro\-lla\-bi\-li\-ty and the Kalman condition:
\begin{proposition}
\label{Prop2}
If any two of the properties \eqref{UC1}, \eqref{UC2} and \eqref{BG} hold, the remaining one is also true.
\end{proposition}

In order to pave the way for detailing the close relationship between observability and controllability, which is well known in control theory, we consider the \textit{adjoint system} and the \textit{dual system} (see \cite{Callier,IlchmannNote} for more details) of 
\eqref{control} described by:% \begin{subequations}
%   \begin{empheq}[left=\empheqlbrace]{align}
\begin{equation}
\label{traspuesto1a}
\dot{x}(t)=-A^{T}(t)x(t)-C^{T}(t)u(t), 
    % & y(t)=B^{T}(t)x(t), \label{traspuesto1b}
  % \end{empheq}
% \end{subequations}
\end{equation}
\begin{equation}
    \label{dual1a}
\dot{x}(t)=A^{T}(-t)x(t)+C^{T}(-t)u(t),
\end{equation}
with $t\in \mathbb{R}$, $x(t)\in\mathbb{R}^{n}$, $u(t)\in\mathbb{R}^{m}$. The transition matrices of the plants of \eqref{traspuesto1a} and \eqref{dual1a} are denoted by $\Psi^{a}(\cdot,\cdot)$ and $\Psi^{d}(\cdot,\cdot)$ respectively, which satisfy the following relation with the original transition matrix of the system \eqref{lin}:
\begin{equation}
\label{Psiadjunto}
    \Psi^{a}(t,\tau)=\Phi_{A}^{T}(\tau,t),
\end{equation}
\begin{equation}
\label{Psidual}
    \Psi^{d}(t,\tau)=\Phi_{A}^{T}(-\tau,-t).
\end{equation}
\begin{comment}
It is important to emphasize that depending on the domain $J$ imposed for the control system \eqref{control1a}-\eqref{control1b}, the domain for the adjoint and dual system will be determined. For example, in the case where $J=\mathbb{R}_0^{+}$, we will have that the domains for the adjoint and dual systems will be $J_a=\mathbb{R}_{0}^{+}$ and $J_d=\mathbb{R}_{0}^{-}$ respectively. At some point in this paper we will consider the case where $J=J_a=J_d=\mathbb{R}$.
\end{comment}
% In fact, we have the following calculation that justifies it:
% $$\begin{array}{rcl}
% \displaystyle\frac{d}{dt}(\Psi_{-A^{T}}(t,\tau))&=&\displaystyle\frac{d}{dt}(\Phi_{A}^{T}(\tau,t))=\frac{d}{dt}(\Phi_{A}(\tau,t))^{T}=-(\Phi_{A}(\tau,t)A(t))^{T},\\
% &=&-A^{T}(t)\Phi_{A}^{T}(\tau,t)
% =-A^{T}(t)\Psi_{-A^{T}}(t,\tau).
% \end{array}$$

The following result establishes the relationship between the Gramian matrices of the system \eqref{control} and the respective Gramian matrices of the systems \eqref{traspuesto1a} and \eqref{dual1a} (see \cite[Theorem 3]{HM} for the proof).

\begin{proposition} 
\label{adjuntoydual}
% Assume that the system \eqref{control1a} admits nonuniform bounded growth. 
By considering $W^{a}$, $K^a$, $W^{d}$ and $K^d$ as the Gramian matrices of controllability of the systems \eqref{traspuesto1a} and \eqref{dual1a} respectively, then the following identities related to the Gramian matrices of observability $M$ and $N$ of the system \eqref{control} are satisfied:
\begin{equation} 
\label{identidadM}
M(t,t+\sigma)=W^{a}(t,t+\sigma)=K^{d}(-t-\sigma,-t).
\end{equation}
\begin{equation}
\label{identidadN}
 N(t,t+\sigma)=K^{a}(t,t+\sigma)=W^{d}(-t-\sigma,-t).
\end{equation}
\end{proposition}

Based on Proposition \ref{adjuntoydual}, we have the ability to identify controllability and observability results depending on the system under consideration. This result allows us to obtain the well known relationship between controllability and observability in the uniform context:   
\begin{theorem} 
\label{adjuntoU}
% Assume that the system \eqref{control1a} admits non-uniform bounded growth. 
    The system \eqref{control} is \textnormal{UCO} if and only if the adjoint system \eqref{traspuesto1a} is \textnormal{UCC}. 
\end{theorem}

\begin{theorem} 
\label{dualU}
% Assume that the system \eqref{control1a} admits non-uniform bounded growth. 
    The system \eqref{control} is \textnormal{UCO} if and only if the dual system \eqref{dual1a} is \textnormal{UCC}. 
\end{theorem}

With this same approach, it is straightforward to obtain a result such as the one stated in Proposition \ref{Prop2}, where the Kalman condition is also related to uniform complete observability (see \cite{Anderson69,Kalman} for details). 

\begin{theorem}
\label{p2l3obs}
    In the uniform context, any two of the properties \eqref{UCO1}, \eqref{UCO2} and \eqref{BG} imply the third one.
\end{theorem}

\section{Non-uniform framework for the controllability and observability}

In order to begin this section, we propose the following motivating linear system
\begin{equation}
\label{EjemploCrecimientoAcotadoNoUniforme}
\dot{x}=-x\sin(x),
\end{equation}
which does not verify the Kalman condition (see \cite[Example 4]{HM} for a detailed proof) and by considering Theorem \ref{p2l3obs}, we can ensure that the system
$$\left\{\begin{array}{rcl}
     \dot{x}&=&-x\sin(x),\\
     y&=&x, \end{array}\right.
     $$
is not uniformly completely observable. It is natural to think of a property that extends to the non-uniform case. This new concept is known as \textit{non-uniform bounded growth} and its definition is detailed below.

\begin{definition} 
 \label{NUBGT}
The linear system \eqref{lin} has a non-uniform bounded growth on $\mathbb{R}$ if there exist constants $K_{0}>0$, $a>0$ and $\eta>0$ such that its transition matrix satisfies
\begin{equation}
\label{Phibound}
\|\Phi_{A}(t,\tau)\|\leq K_{0}e^{\eta |\tau| }e^{a|t-\tau|}\quad \textnormal{for any $t,\tau\in \mathbb{R}$}.
\end{equation}
\end{definition}

We emphasize that in \cite[Example 4]{HM} it was proved that system \eqref{EjemploCrecimientoAcotadoNoUniforme} verifies this property. With this in mind, it makes sense to ask about the existence of new controllability and observability properties that adapt to non-uniform bounded growth condition. We call them \textit{non-uniform complete controllability} (NUCC) and \textit{non-uniform complete observability} (NUCO). The NUCC concept was presented in \cite{HMR}, while the NUCO idea was developed in \cite{HM}. Both are directly related via duality and both share a strong relation with a non-uniform form of the Kalman condition. Here, we put all the re\-le\-vant results together, in order to move towards the main goals of the article.\\   

\begin{definition}
\label{NUIMBU}
The control system \eqref{control1a} has the \textit{non-uniform Kalman condition} if there exist $\nu>0$ and a function  $\alpha(\cdot)\in\mathcal{B}$ satisfying
\begin{equation}
\label{crec-acot}
\|\Phi_{A}(t,\tau)\|\leq e^{\nu\tau}\,\alpha(|t-\tau|)  \quad \textnormal{for any $t,\tau\in\mathbb{R}$}.
\end{equation}
\end{definition}

\begin{remark}
\label{Kalmanproperty}
   We note that non-uniform bounded growth is a particular case of the non-uniform Kalman property.
\end{remark}

\begin{definition}
\label{DEFNUCC}
The linear control system \eqref{control1a} is said to be non-uniformly completely controllable \textnormal{(NUCC)} on $\mathbb{R}$ if there exist fixed numbers $\mu_0\geq0$, $\mu_1\geq0$, $\tilde{\mu}_0\geq0$, $\tilde{\mu}_1\geq0$ and functions $\alpha_0$, $\beta_{0}$, $\alpha_1$, $\beta_1:[0,+\infty)\to(0,+\infty)$ such that for any $t\in \mathbb{R}$, there exists $\sigma_0(t)>0$ with:
\begin{equation}
\label{alpha0alpha1}
    0<e^{-2\mu_0 t}\alpha_0(\sigma)I\leq W(t,t+\sigma)\leq e^{2\mu_1 t}\alpha_1(\sigma)I,
\end{equation}
\begin{equation}
\label{K}
    0<e^{-2\tilde{\mu}_0 t}\beta_0(\sigma) I\leq K(t,t+\sigma)\leq e^{2\tilde{\mu}_1 t}\beta_1(\sigma) I, 
\end{equation}
for every $\sigma\geq \sigma_{0}(t)$.
\end{definition}

\begin{definition}
% \label{NUCO}
The linear control system \eqref{control} is said to be non-uniformly completely observable\textnormal{(NUCO)} on $\mathbb{R}$ if there exist fixed numbers $\nu_0\geq0$, $\nu_1\geq0$, $\bar{\nu}_0\geq0$, $\bar{\nu}_1\geq0$ and functions $\vartheta_0$, $\varrho_0$, $\vartheta_1$, $\varrho_1:[0,+\infty)\to (0,+\infty)$ such that for any $t\in \mathbb{R}$, there exists $\sigma_0(t)>0$ with:
\begin{equation}
\label{Mobserv}
e^{-2\nu_0 |t|}\vartheta_0(\sigma)I\leq M(t,t+\sigma)\leq  e^{2\nu_1 |t|}\vartheta_1(\sigma)I
\end{equation}
\begin{equation}
\label{Mobserv2}
e^{-2\bar{\nu}_0 |t|}\varrho_0(\sigma)I\leq N(t,t+\sigma)\leq  e^{2\bar{\nu}_1 |t|}\varrho_1(\sigma)I
\end{equation}
for every $\sigma\geq\sigma_0(t)$.
\end{definition}

By considering Proposition \ref{adjuntoydual}, we obtain the following results that relate the controllability and observability of the control system and its adjoint and dual system in the non-uniform framework.

\begin{theorem} 
\label{adjuntoNU}
% Assume that the system \eqref{control1a} admits non-uniform bounded growth. 
    The system \eqref{control} is \textnormal{NUCO} if and only if the adjoint system \eqref{traspuesto1a} is \textnormal{NUCC}. 
\end{theorem}

\begin{theorem} 
\label{dualNU}
% Assume that the system \eqref{control1a} admits non-uniform bounded growth. 
    The system \eqref{control} is \textnormal{NUCO} if and only if the dual system \eqref{dual1a} is \textnormal{NUCC}. 
\end{theorem}

Based on the previous results, thanks to the dual relationship between non-uniform complete controllability and non-uniform complete observability, mixed with the non-uniform Kalman condition, the following results are guaranteed (see \cite{HM,HMR}).

\begin{theorem}
\label{p2l3obsNOUNIFORME}
    Any two of the properties \eqref{crec-acot}, \eqref{Mobserv} and \eqref{Mobserv2} imply the third one.
\end{theorem}

\begin{theorem}
\label{p2l3}
Any two of the properties \eqref{crec-acot}, \eqref{alpha0alpha1} and \eqref{K}
imply the third.
\end{theorem}

The non-uniform complete controllable systems are a new class of time-varying linear control systems lying strictly between uniform complete controllable and complete controllable linear systems, and the same fact is true for non-uniform complete observable systems. Several illustrative examples of concrete control systems belonging to these particular classes can be found in \cite{HM,HMR}.

\section{Observability conservation by output feedback in a non-uniform context} 
 In this section we will study the invariance of the non-uniform complete observability property when we consider an output feedback. The ideas and results that we will see below are strongly based on \cite{Zhang2015ObservabilityCB}, which are adapted to the non-uniform setting of bounded growth, uniform complete observability and their respective bounds.

In order to begin, consider the control system \eqref{control}.
% \begin{subequations}
%   \begin{empheq}[left=\empheqlbrace]{align}
%     & \dot{x}(t)=A(t)x(t)+B(t)u(t), \label{LTVsystem1a} \\
%     & y(t)=C(t)x(t), \label{LTVsystem1b}
%   \end{empheq}
% \end{subequations}
% where $A(t)\in M_{n\times n}(\mathbb{R})$, $B(t)\in M_{n\times p}(\mathbb{R})$ and $C(t)\in M_{m\times n}(\mathbb{R})$. 
Moreover, by using an output feedback $u(t)=-F(t)y(t)$, for some matrix $F(\cdot)\in M_{p\times m}(\mathbb{R})$, and by defining $K(t):=B(t)F(t)$, then we replace this output feedback in \eqref{control} and we obtain the following two systems:

$$
\left\{\begin{array}{lcl}
\dot{x}(t)&=&A(t)x(t)-K(t)y(t),\\
y(t)&=&C(t)x(t),
\end{array}
\right.
$$
% \begin{subequations}
%   \begin{empheq}[left=\empheqlbrace]{align}
%      & \dot{x}(t)=A(t)x(t)-K(t)y(t),\\ 
%     % \label{ClosedLoopa} \\
%     & y(t)=C(t)x(t), 
%     % \label{ClosedLoopb}
%   \end{empheq}
% \end{subequations}
% \begin{equation}
% \label{ClosedLoop}
% \begin{array}{rcl}
% \dot{x}(t)&=&A(t)x(t)-L(t)y(t)\\
% y(t)&=&C(t)x(t)
% \end{array}
% \end{equation}
and
\begin{subequations}\label{controlA-KC}
  \begin{empheq}[left=\empheqlbrace]{align}
    & \dot{x}(t)=(A(t)-K(t)C(t))x(t), \label{ClosedLoop2a} \\
    & y(t)=C(t)x(t), \label{ClosedLoop2b}
  \end{empheq}
\end{subequations}
% \begin{equation}
% \label{ClosedLoop2}
% \begin{array}{rcl}
% \dot{x}(t)&=&(A(t)-L(t)C(t))x(t)\\
% y(t)&=&C(t)x(t)
% \end{array}
% \end{equation}
% which are equivalents. In this way, the system \eqref{ClosedLoop2a}-\eqref{ClosedLoop2b}
% % , and additionally \eqref{ClosedLoop2a}-\eqref{ClosedLoop2b}, 
% has the same property of observability than \eqref{LTVsystem1a}-\eqref{LTVsystem1b}.

In order to pave the way towards the main result, we will show the fo\-llo\-wing technical result, which ensures that the non-uniform bounded growth property is preserved, when we consider forward time, in a closed-loop system.

\begin{lemma}
    \label{RobustezNUBGLineal}
    Suppose that the system \eqref{lin} admits a non-uniform bounded growth with constants $K_0$, $a$, and $\varepsilon$. In addition, consider the nonlinear and nonautonomous perturbed system
    \begin{equation}
        \label{PerturbacionNoLineal}
        \dot{x}(t)=(A(t)+P(t))x(t)
    \end{equation}
    and assume that there exist $\mathcal{P}>0$, $p>0$ such that for any $t\in\mathbb{R}$:
    \begin{equation}
    \label{CotafLineal}
        \|P(t)\|\leq \mathcal{P}e^{-p|t|}.
    \end{equation}

    If the parameters $\varepsilon$ and $p$ satisfy
    \begin{equation}
        \label{CondicionparametrosNoLineal}
        p>\varepsilon
    \end{equation}
    then the system \eqref{PerturbacionNoLineal} admits non-uniform bounded growth property with cons\-tants $K_0$, $a+K_0\mathcal{P}$ and $\varepsilon$.
\end{lemma}

\begin{proof}
    The transition matrix of the system \eqref{PerturbacionNoLineal}, denoted by $\tilde{\Phi}(\cdot,\cdot)$,
    is given by 
    \begin{equation} 
    \label{SolucionA-KC}
    \tilde{\Phi}(t_2,t_1)=\Phi_{A}(t_2,t_1)+\int_{t_1}^{t_2}\Phi_{A}(t_2,s)P(s)\tilde{\Phi}(s,t_1)\; ds
    \end{equation}
    and by assuming the non-uniform bounded growth property of the system \eqref{lin} and the conditions \eqref{CotafLineal}-\eqref{CondicionparametrosNoLineal}, we have the following estimate for $t_2\geq t_1$:
    $$
    \begin{array}{rcl}
    \left\|\tilde{\Phi}(t_2,t_1)\right\|
    &\leq&\displaystyle K_0e^{a(t_2-t_1)+\varepsilon |t_1|}+\int_{t_1}^{t_2}K_0e^{a(t_2-s)+\varepsilon |s|}\mathcal{P}e^{-p |s|}\left\|\tilde{\Phi}(s,t_1)\right\|\; ds,\\
    &\leq&\displaystyle K_0e^{a(t_2-t_1)+\varepsilon |t_1|}+\int_{t_1}^{t_2}K_0\mathcal{P}e^{a(t_2-s)}\left\|\tilde{\Phi}(s,t_1)\right\|\; ds\\
    \end{array}
    $$
    and multiplying the previous inequality by $e^{-a(t_2-t_1)}$:
    $$ 
    \begin{array}{rcl}
    e^{-a(t_2-t_1)}\left\|\tilde{\Phi}(t_2,t_1)\right\|
    % &\leq&\displaystyle K_0e^{\varepsilon |t_1|}+\int_{t_1}^{t_2}K_0\mathcal{P}e^{(\varepsilon-p)|s|}e^{-a(s-t_1)}\left\|\tilde{\Phi}(s,t_1)\right\|\; ds,\\
    &\leq&\displaystyle K_0e^{\varepsilon |t_1|}+\int_{t_1}^{t_2}K_0\mathcal{P}e^{-a(s-t_1)}\left\|\tilde{\Phi}(s,t_1)\right\|\; ds.
    \end{array}
    $$
By considering the Gronwall's inequality we obtain that:
$$\begin{array}{rcl}
e^{-a(t_2-t_1)}\left\|\tilde{\Phi}(t_2,t_1)\right\|&\leq&\displaystyle K_0e^{\varepsilon |t_1|}e^{\int_{t_1}^{t_2}K_0\mathcal{P}\; ds},\\
\left\|\tilde{\Phi}(t_2,t_1)\right\|&\leq&\displaystyle K_0e^{a(t_2-t_1)+\varepsilon |t_1|}e^{K_0\mathcal{P}(t_2-t_1)},\\
&=&K_0e^{(a+K_0\mathcal{P})(t_2-t_1)+\varepsilon|t_1|}.
\end{array}$$

On the other hand, if now we consider $t_2\leq t_1$, then \eqref{SolucionA-KC} is transformed into

\begin{equation*} 
    % \label{SolucionA-KC}
    \tilde{\Phi}(t_2,t_1)=\Phi_{A}(t_2,t_1)+\int_{t_2}^{t_1}\Phi_{A}(t_2,s)P(s)\tilde{\Phi}(s,t_1)\; ds
    \end{equation*}

    After that, by using the non-uniform bounded growth property and \eqref{CotafLineal}-\eqref{CondicionparametrosNoLineal}, then for this case we have that:

$$
    \begin{array}{rcl}
    \left\|\tilde{\Phi}(t_2,t_1)\right\|
    &\leq&\displaystyle K_0e^{a(t_1-t_2)+\varepsilon |t_1|}+\int_{t_2}^{t_1}K_0e^{a(s-t_2)+\varepsilon |s|}\mathcal{P}e^{-p |s|}\left\|\tilde{\Phi}(s,t_1)\right\|\; ds,\\
    &\leq&\displaystyle K_0e^{a(t_1-t_2)+\varepsilon |t_1|}+\int_{t_2}^{t_1}K_0\mathcal{P}e^{a(s-t_2)}\left\|\tilde{\Phi}(s,t_1)\right\|\; ds
    \end{array}
    $$
    and now multiplying the previous inequality by $e^{-a(t_1-t_2)}$:
    $$
    \begin{array}{rcl}
    e^{-a(t_1-t_2)}\left\|\tilde{\Phi}(t_2,t_1)\right\|
    % &\leq&\displaystyle K_0e^{\varepsilon |t_1|}+\int_{t_1}^{t_2}K_0\mathcal{P}e^{(\varepsilon-p)|s|}e^{-a(t_1-s)}\left\|\tilde{\Phi}(s,t_1)\right\|\; ds,\\
    &\leq&\displaystyle K_0e^{\varepsilon |t_1|}+\int_{t_1}^{t_2}K_0\mathcal{P}e^{-a(t_1-s)}\left\|\tilde{\Phi}(s,t_1)\right\|\; ds.
    \end{array}
    $$

Just like before, by using the Gronwall inequality we obtain that:
$$\begin{array}{rcl}
e^{-a(t_1-t_2)}\left\|\tilde{\Phi}(t_2,t_1)\right\|&\leq&\displaystyle K_0e^{\varepsilon |t_1|}e^{\int_{t_2}^{t_1}K_0\mathcal{P}\; ds},\\
\left\|\tilde{\Phi}(t_2,t_1)\right\|&\leq&\displaystyle K_0e^{a(t_1-t_2)+\varepsilon |t_1|}e^{K_0\mathcal{P}(t_1-t_2)},\\
&=&K_0e^{(a+K_0\mathcal{P})(t_1-t_2)+\varepsilon|t_1|},
\end{array}$$
which conclude the proof.
\end{proof}

Based on this result, in order to prove the main theorem, we establish a direct consequence that adapts to the problem addressed in this work, namely, the preservation of the non-uniform bounded growth property for a linear perturbation.
\begin{lemma}
\label{RobustezNUBG}
    Suppose that the plant of the system \eqref{lin} admits a non-uniform bounded growth with constants $K_0$, $a$, and $\varepsilon$. In addition, by considering the system \eqref{ClosedLoop2a}, suppose that there exist $\delta>0$ and $\gamma\geq0$ such that for any $z\in \mathbb{R}$ we have that:
    \begin{equation}
    \label{Kcondition}
    \|K(z)\|\leq \mathcal{K}e^{-\delta |z|},
    \end{equation}
    \begin{equation} 
    \label{Ccondition}
    \|C(z)\|\leq \mathcal{C}e^{\gamma |z|}.
    \end{equation}
    If we have the following condition about the constants
    \begin{equation}
        \label{conditionGED}
    \delta>\gamma+\varepsilon,
    \end{equation}
    then the transition matrix of system \eqref{ClosedLoop2a}, with
    \begin{equation}
        \label{ATilde}
        \tilde{A}(t)=A(t)-K(t)C(t),
    \end{equation}
     satisfies the non-uniform bounded growth property with constants $K_0$, $a+K_0\mathcal{K}\mathcal{C}$ and $\varepsilon$. In other words, for any $t_1, t_2\in\mathbb{R}$ and if $\Phi_{\tilde{A}}(\cdot,\cdot)$ is the transition matrix of \eqref{ClosedLoop2a} we have that:
     $$\|\Phi_{\tilde{A}}(t_2,t_1)\|\leq K_0e^{(a+K_0\mathcal{K}\mathcal{C})|t_2-t_1|+\varepsilon|t_1|}.$$
\end{lemma}
\begin{proof}
    By using the Lemma \ref{RobustezNUBGLineal} with $P(t)=K(t)C(t)$, we can obtain the estimates \eqref{CotafLineal} and \eqref{CondicionparametrosNoLineal} by considering \eqref{Kcondition}, \eqref{Ccondition} and \eqref{conditionGED}, where $p=\delta-\gamma>0$ and $\mathcal{P}=\mathcal{K}\mathcal{C}$.
    \end{proof}

The following theorem establishes the preservation of non-uniform complete observability by output feedback, whose proof follows the main result of \cite{Zhang2015ObservabilityCB}.

\begin{theorem}
\label{observableconK}
    Suppose that the system \eqref{lin} admits non-uniform bounded growth with constants $K_0$, $a$ and $\varepsilon$. Furthermore, if we assume conditions \eqref{Kcondition}, \eqref{Ccondition} and
 \eqref{conditionGED} 
    % \begin{equation}
    %     \label{conditionAED}
    %     \delta-\varepsilon>0
    % \end{equation}
    are satisfied, we have that the system \eqref{control} is non-uniformly completely observable if and only if the system \eqref{controlA-KC}\footnote{\eqref{controlA-KC} denotes the control system \eqref{ClosedLoop2a}-\eqref{ClosedLoop2b}} is non-uniformly completely observable.
\end{theorem} 

\begin{remark}
    Before proceeding with the proof of the theorem, we can note that there exists a kind of ``symmetry" between the matrices $A(t)$ and $\tilde{A}(t)$ described in \eqref{ATilde},    
due to the fact that it can be rewritten as follows
    \begin{equation*}
        % \label{AbasadoAtilde}
        A(t)=\tilde{A}(t)-(-K(t))C(t).
    \end{equation*}
    Based on the above, the implication from right to left in Theorem \textnormal{\ref{observableconK}} can be proved based on the implication from left to right.
\end{remark}
\begin{proof}
    By considering the parameter $(s,t)$ instead of $(t_2,t_1)$ in \eqref{SolucionA-KC}, and consequently $s\geq t$, we have that
    $$\Phi_{\tilde{A}}(s,t)-\Phi_{A}(s,t)=-\displaystyle\int_{t}^{s}\Phi_{A}(s,p)K(p)C(p)\Phi_{\tilde{A}}(p,t)\; dp.$$

    Multiplying from the left on both sides of the previous equality by $C(s)$, multiplying from the right by any unit vector $v\in\mathbb{R}^{n}$ and by taking the integral of the squared norm with respect to $s$ on both sides, then we have the following expression:
    $$\begin{array}{rcl}
    &&\displaystyle \int_{t}^{t+\sigma}\left \| C(s)(\Phi_{\tilde{A}}(s,t)-\Phi_{A}(s,t))v\right \|^{2}ds\\
    &=&\displaystyle \int_{t}^{t+\sigma}\left \| C(s)\int_{t}^{s}\Phi_{A}(s,p)K(p)C(p)\Phi_{\tilde{A}}(p,t)v\;dp\right \|^{2}ds.
    \end{array}
    $$
    
    By developing the squared Euclidean norm at the left hand side and by reordering the terms, we have that
    \begin{equation}
    \label{igualdadinicial}
    \begin{array}{rcl}
    &&\displaystyle \int_{t}^{t+\sigma}\left \| C(s)\Phi_{A}(s,t)v\right \|^{2}ds
    =-\displaystyle \int_{t}^{t+\sigma}\left \| C(s)\Phi_{\tilde{A}}(s,t)v\right \|^{2}ds\\
    &+&\displaystyle\int_{t}^{t+\sigma}2v^{T}\Phi_{A}^{T}(s,t)C^{T}(s)C(s)\Phi_{\tilde{A}}(s,t)v\;ds\\
    &+& \displaystyle \int_{t}^{t+\sigma}\left \| C(s)\int_{t}^{s}\Phi_{A}(s,p)K(p)C(p)\Phi_{\tilde{A}}(p,t)vdp\right \|^{2} ds.
    \end{array}\end{equation}

    Now, each terms described in the previous equality will be properly bounded in order to achieve the non-uniform complete observability of the system \eqref{controlA-KC}. 

    Firstly, if the system \eqref{control} is non-uniformly completely observable, from \eqref{gramianobservability} we have that for any unitary vector $v\in\mathbb{R}^{n}$:
    \begin{equation} 
    \label{MTeorema}
    \begin{array}{rcl}
    \vartheta_0(\sigma)e^{-2\nu_0 |t|}&\leq& v^{T}M_{A}(t,t+\sigma)v\\
&=&\displaystyle\int_{t}^{t+\sigma}\left \|C(s)\Phi_{A}(s,t)v \right \|^{2}ds,
    \leq 
    %v^{T}\vartheta_1(\sigma)e^{2\nu_1 |t|}v
    \vartheta_1(\sigma)e^{2\nu_1 |t|}.
    \end{array}
    \end{equation}
From this point forward, we define the rest of terms that appear in \eqref{igualdadinicial} to facilitate the writing of the rest of the proof:
    \begin{equation} 
    \label{lambda}
    v^{T}M_{\tilde{A}}(t,t+\sigma)v=\displaystyle\int_{t}^{t+\sigma}\left\|C(s)\Phi_{\tilde{A}}(s,t)v \right \|^{2}ds,
    % =v^{T}M(t,t+\sigma)v.
    \end{equation}
    \begin{equation}
        \label{lambda_2}
\Lambda(t,t+\sigma)=\displaystyle\int_{t}^{t+\sigma}2v^{T}\Phi_{A}^{T}(s,t)C^{T}(s)C(s)\Phi_{\tilde{A}}(s,t)v\;ds,
    \end{equation}
    \begin{equation}
        \label{Delta}
        \Delta(t,t+\sigma)=\displaystyle \int_{t}^{t+\sigma}\left \| C(s)\int_{t}^{s}\Phi_{A}(s,p)K(p)C(p)\Phi_{\tilde{A}}(p,t)vdp\right \|^{2} ds.
    \end{equation}
    
    On the other hand, by using the Cauchy-Schwarz inequality and by considering \eqref{MTeorema} and \eqref{lambda}, then \eqref{lambda_2} can be bounded as follows:
    \begin{equation} \label{CSdesigualdad}
    \begin{array}{rcl}&&\Lambda(t,t+\sigma)\\
    &\leq& \displaystyle 2\sqrt{\int_{t}^{t+\sigma}\left\|C(s)\Phi_{A}(s,t)v \right \|^{2}ds\int_{t}^{t+\sigma}\left\|C(s)\Phi_{\tilde{A}}(s,t)v \right \|^{2}ds},\\
    &\leq& 2\sqrt{\vartheta_1(\sigma)e^{2\nu_1 |t|}}\sqrt{v^{T}M_{\tilde{A}}(t,t+\sigma)v}.
    \end{array}
    \end{equation}

Therefore, by mixing the previous estimates, Cauchy-Schwarz inequality and by considering condition \eqref{Ccondition}, we have that:
\begin{equation}
\label{acotandoDelta}
\begin{array}{rcl}
&&\Delta(t,t+\sigma)\\
&\leq&\displaystyle\int_{t}^{t+\sigma}\|C(s)\|^{2}\left\|\int_{t}^{s}\Phi_{A}(s,p) K(p) C(p)\Phi_{\tilde{A}}(p,t)vdp \right\|^{2}ds,\\
&\leq&\displaystyle\int_{t}^{t+\sigma}\mathcal{C}^{2}e^{2\gamma |s|}\left(\int_{t}^{s}\left\|\Phi_{A}(s,p) \right \|^{2}\left \|K(p) \right \|^{2}dp \right)\left (\int_{t}^{s}\left\| C(p)\Phi_{\tilde{A}}(p,t)v\right\|^{2}dp \right)ds.\\
% &\leq&v^{T}M_{\tilde{A}}(t,t+\sigma)v\displaystyle\int_{t}^{t+\sigma}\frac{K_0\mathcal{K}\mathcal{C}^{2}}{2(a+\delta-\varepsilon)}e^{-2(\delta-\varepsilon-\gamma)|s|}e^{2(a+\delta-\varepsilon)(s-t)}ds,\\
% &\leq&\displaystyle v^{T}M_{\tilde{A}}(t,t+\sigma)v\frac{K_0\mathcal{K}\mathcal{C}^{2}e^{2(a+\delta-\varepsilon)\sigma}}{4(a+\delta-\varepsilon)^2}.
\end{array}
\end{equation}

After that, by using non-uniform bounded growth property, \eqref{Kcondition} and \eqref{conditionGED} ($\delta>\varepsilon$), we can bound the central factor of \eqref{acotandoDelta}:
$$\begin{array}{rcl}
&&\displaystyle \int _{t}^{s}\left\| \Phi_{A}(s,p)\right \|^{2}\left\| K(p)\right\|^{2}dp\leq \displaystyle\int_{t}^{s}K_0e^{2a(s-p)+2\varepsilon |p|}\mathcal{K}e^{-2\delta |p|}dp,\\
&=&\displaystyle\int_{t}^{s}K_0\mathcal{K}e^{2a(s-p)}e^{-2(\delta-\varepsilon) |p|}dp,
\leq\displaystyle\int_{t}^{s}K_0\mathcal{K}e^{2(a+\delta-\varepsilon)(s-p)}e^{-2(\delta-\varepsilon) |s|}dp,\\
&=&\displaystyle \frac{K_0\mathcal{K}}{2(a+\delta-\varepsilon)}e^{-(\delta-\varepsilon)|s|}\left [-1+ e^{2(a+\delta-\varepsilon)(s-t)}\right ],\\
&\leq&\displaystyle \frac{K_0\mathcal{K}}{2(a+\delta-\varepsilon)}e^{-2(\delta-\varepsilon)|s|}e^{2(a+\delta-\varepsilon)(s-t)}
\end{array}$$
and by using this estimate in \eqref{acotandoDelta} we obtain:

$$\begin{array}{rcl}
\Delta(t,t+\sigma)
% &\leq&\displaystyle\int_{t}^{t+\sigma}\|C(s)\|^{2}\left\|\int_{t}^{s}\Phi_{A}(s,p) K(p) C(p)\Phi_{\tilde{A}}(p,t)vdp \right\|^{2}ds,\\
% &\leq&\displaystyle\int_{t}^{t+\sigma}\mathcal{C}^{2}e^{2\gamma |s|}\left(\int_{t}^{s}\left\|\Phi_{A}(s,p) \right \|^{2}\left \|K(p) \right \|^{2}dp \right)\left (\int_{t}^{s}\left\| C(p)\Phi_{\tilde{A}}(p,t)v\right\|^{2}dp \right)ds,\\
&\leq&v^{T}M_{\tilde{A}}(t,t+\sigma)v\displaystyle\int_{t}^{t+\sigma}\frac{K_0\mathcal{K}\mathcal{C}^{2}}{2(a+\delta-\varepsilon)}e^{-2(\delta-\varepsilon-\gamma)|s|}e^{2(a+\delta-\varepsilon)(s-t)}ds,\\
&\leq&\displaystyle v^{T}M_{\tilde{A}}(t,t+\sigma)v\frac{K_0\mathcal{K}\mathcal{C}^{2}e^{2(a+\delta-\varepsilon)\sigma}}{4(a+\delta-\varepsilon)^2}.
\end{array}$$

Based on the previous estimations, the equality 
\eqref{igualdadinicial} is expressed in terms of the following inequality:

\begin{equation}
\label{desigualdadlambda}
\begin{array}{rcl}
&&\vartheta_0(\sigma)e^{-2\nu_0 |t|}\\
&\leq& -v^{T}M_{\tilde{A}}(t,t+\sigma)v+2\sqrt{\vartheta_1(\sigma)e^{2\nu_1 |t|}}\sqrt{v^{T}M_{\tilde{A}}(t,t+\sigma)v}\\
&+&\varphi(\sigma)v^{T}M_{\tilde{A}}(t,t+\sigma)v,\\
&=&(\varphi(\sigma)-1)v^{T}M_{\tilde{A}}(t,t+\sigma)v\\
&+&2\sqrt{\vartheta_1(\sigma)e^{2\nu_1 |t|}}\sqrt{v^{T}M_{\tilde{A}}(t,t+\sigma)v},
\end{array}
\end{equation}
where 
$$\varphi(\sigma)=\frac{K_0\mathcal{K}\mathcal{C}^{2}e^{2(a+\delta-\varepsilon)\sigma}}{4(a+\delta-\varepsilon)^2}.
$$

Based on the above, we will analyze three cases: 
\begin{itemize}
    \item [a)] If $\varphi(\sigma)=1$, then from \eqref{desigualdadlambda} we have that:
$$
\begin{array}{rcl}
\vartheta_0(\sigma)e^{-2\nu_0 |t|}&\leq& 2\sqrt{\vartheta_1(\sigma)e^{2\nu_1 |t|}}\sqrt{v^{T}M_{\tilde{A}}(t,t+\sigma)v},\\
\displaystyle\frac{\vartheta^{2}_0(\sigma)}{4\vartheta_1(\sigma)}e^{-(4\nu_0+2\nu_1)|t|}&\leq& v^{T}M_{\tilde{A}}(t,t+\sigma)v.
\end{array}
$$

\item[b)] If $\varphi(\sigma)>1$, then we rewrite \eqref{desigualdadlambda} as follows:

\begin{equation}
\label{eq:cuadratica}
\begin{array}{rcl}
&&(\varphi(\sigma)-1)v^{T}M_{\tilde{A}}(t,t+\sigma)v\\
&+&2\sqrt{\vartheta_1(\sigma)e^{2\nu_1 |t|}}\sqrt{v^{T}M_{\tilde{A}}(t,t+\sigma)v}-\vartheta_0(\sigma)e^{-2\nu_0 |t|}\geq0,
\end{array}
\end{equation}
and in order to solve this inequality, we rationalize and develop the rest by considering $\varphi(\sigma)-1\neq0$:
\begin{equation}
\label{eq:solucioncuadratica}
\begin{array}{rcl}
\displaystyle\sqrt{v^{T}M_{\tilde{A}}(t,t+\sigma)v}&\geq&\displaystyle\frac{-\sqrt{\vartheta_1(\sigma)e^{2\nu_1 |t|}}+\sqrt{\vartheta_1(\sigma)e^{2\nu_1 |t|}+(\varphi(\sigma)-1)\vartheta_0(\sigma)e^{-2\nu_0 |t|}}}{\varphi(\sigma)-1},\\
\displaystyle \sqrt{v^{T}M_{\tilde{A}}(t,t+\sigma)v}&\geq&\displaystyle\frac{\vartheta_0(\sigma)e^{-2\nu_0 |t|}}{\left(\sqrt{\vartheta_1(\sigma)e^{2\nu_1 |t|}}+\sqrt{\vartheta_1(\sigma)e^{2\nu_1 |t|}+(\varphi(\sigma)-1)\vartheta_0(\sigma)e^{-2\nu_0 |t|}}\right)},\\
\displaystyle v^{T}M_{\tilde{A}}(t,t+\sigma)v&\geq&\displaystyle\left(\frac{\vartheta_0(\sigma)e^{-2\nu_0 |t|}}{\left(\sqrt{\vartheta_1(\sigma)e^{2\nu_1 |t|}}+\sqrt{\vartheta_1(\sigma)e^{2\nu_1 |t|}+(\varphi(\sigma)-1)\vartheta_0(\sigma)e^{-2\nu_0 |t|}}\right)}\right)^{2},\\
% \displaystyle v^{T}M_{\tilde{A}}(t,t+\sigma)v&\geq&\displaystyle\left(\frac{\vartheta_0(\sigma)e^{-2\nu_0 |t|}}{\sqrt{\vartheta_1(\sigma)e^{2\nu_1 |t|}}+\sqrt{\vartheta_1(\sigma)e^{2\nu_1 |t|}+(\varphi(\sigma)-1)\vartheta_0(\sigma)e^{-2\nu_0 |t|}}}\right)^{2}.
\end{array}
\end{equation}
which implies that
    $$\begin{array}{rcl}
     v^{T}M_{\tilde{A}}(t,t+\sigma)v&\geq&\displaystyle \left(\frac{\vartheta_0(\sigma)e^{-2\nu_0 |t|}}{2\sqrt{\max\{\vartheta_1(\sigma),(\varphi(\sigma)-1)\vartheta_0(\sigma)\}e^{2\nu_1 |t|}}}\right)^{2},\\
v^{T}M_{\tilde{A}}(t,t+\sigma)v&\geq&\displaystyle\frac{\vartheta_0^{2}(\sigma)}{4\max\{\vartheta_1(\sigma),(\varphi(\sigma)-1)\vartheta_0(\sigma)\}}e^{-(4\nu_0+2\nu_1)|t|}.
\end{array}$$
    \item[c)] If $\varphi(\sigma)-1<0$, then we can deduce that based on inequality \eqref{eq:cuadratica}, we obtain the same inequality \eqref{eq:solucioncuadratica}. After that, we have that:
    $$
    \begin{array}{rcl}
    \displaystyle\left ( \frac{\vartheta_0(\sigma)e^{-2\nu_0 |t|}}{2\sqrt{\vartheta_1(\sigma)e^{2\nu_1 |t|}}}\right )^{2}&\leq& v^{T}M_{\tilde{A}}(t,t+\sigma)v,\\
    \displaystyle\frac{\vartheta^{2}_0(\sigma)}{4\vartheta_1(\sigma)}e^{-(4\nu_0+2\nu_1)|t|}&\leq& v^{T}M_{\tilde{A}}(t,t+\sigma)v.
    \end{array}
    $$
\end{itemize}

Now we must establish the upper bound for $v^{T}M_{\tilde{A}}(t,t+\sigma)v$, which will be obtained from a similar argument to the one used to get the lower bound. Based on this kind of symmetry existing between $A(t)$ and $\tilde{A}(t)$, the ``sy\-mme\-tric'' counterpart of (\ref{igualdadinicial}) corresponds to exchange $\Phi_{A}(\cdot,\cdot)$ with $\Phi_{\tilde{A}}(\cdot,\cdot)$ and replacing $K(p)$ by $-K(p)$. In this way, we will have the following expression: 
\begin{equation*}
    % \label{igualdadinicial}
    \begin{array}{rcl}
    &&\displaystyle \int_{t}^{t+\sigma}\left \| C(s)\Phi_{\tilde{A}}(s,t)v\right \|^{2}ds=-\displaystyle \int_{t}^{t+\sigma}\left \| C(s)\Phi_{A}(s,t)v\right \|^{2}ds\\
    &+&\displaystyle\int_{t}^{t+\sigma}2v^{T}\Phi_{\tilde{A}}^{T}(s,t)C^{T}(s)C(s)\Phi_{A}(s,t)v\;ds\\
    &+& \displaystyle \int_{t}^{t+\sigma}\left \| C(s)\int_{t}^{s}\Phi_{\tilde{A}}(s,p)K(p)C(p)\Phi_{A}(p,t)v\;dp\right \|^{2} ds.
    \end{array}\end{equation*}
    
By rewriting the previous equality, we obtain that:
\begin{equation}
    \label{igualdadinicialcaso2}
    \begin{array}{rcl}
&&\displaystyle\int_{t}^{t+\sigma}\left \| C(s)\Phi_{A}(s,t)v\right \|^{2}ds
=-\displaystyle \int_{t}^{t+\sigma}\left \| C(s)\Phi_{\tilde{A}}(s,t)v\right \|^{2}ds\\
&+&\displaystyle\int_{t}^{t+\sigma}2v^{T}\Phi_{\tilde{A}}^{T}(s,t)C^{T}(s)C(s)\Phi_{A}(s,t)v\;ds\\
    &+& \displaystyle \int_{t}^{t+\sigma}\left \| C(s)\int_{t}^{s}\Phi_{\tilde{A}}(s,p)K(p)C(p)\Phi_{A}(p,t)v\;dp\right \|^{2} ds.
    \end{array}\end{equation}

The second term on the right-hand side of \eqref{igualdadinicialcaso2} will be bounded using the Cauchy-Schwarz inequality as we did in the estimate \eqref{CSdesigualdad}. If we define the third term of \eqref{igualdadinicialcaso2} as follows:
    \begin{equation}
        \label{Deltatilde}
        \tilde{\Delta}(t,t+\sigma)=\displaystyle \int_{t}^{t+\sigma}\left \| C(s)\int_{t}^{s}\Phi_{\tilde{A}}(s,p)K(p)C(p)\Phi_{A}(p,t)vdp\right \|^{2} ds,
    \end{equation}
by mixing Cauchy-Schwarz inequality and  \eqref{Ccondition}, then we have that: 
% \eqref{PhiK} and \eqref{Mobserv}, then we have that:
$$\begin{array}{rcl}
&&\tilde{\Delta}(t,t+\sigma)\\&\leq&\displaystyle \int_{t}^{t+\sigma}\|C(s)\|^{2}\left \|\int_{t}^{s}\Phi_{\tilde{A}}(s,p)K(p)C(p)\Phi_{A}(p,t)v\;dp\right \|^{2} ds,\\
&\leq&\displaystyle\int_{t}^{t+\sigma}\mathcal{C}^{2}e^{2\gamma |s|}\left(\int_{t}^{s}\left\|\Phi_{\tilde{A}}(s,p) \right \|^{2}\left \|K(p) \right \|^{2}dp \right)\left (\int_{t}^{s}\left\| C(p)\Phi_{A}(p,t)v\right\|^{2}dp \right)ds.
% &\leq&\displaystyle\int_{t}^{t+\sigma}\frac{K_0\mathcal{K}\mathcal{C}^{2}}{2(a+K_0\mathcal{K}\mathcal{C}+\delta-\varepsilon)}e^{-2(\delta-\varepsilon-\gamma)|s|}e^{2(a+K_0\mathcal{K}\mathcal{C}+\delta-\varepsilon)(s-t)}\vartheta_1(\sigma)e^{2\nu_1 |t|}ds,\\
% &\leq&\displaystyle\frac{\vartheta_1(\sigma)K_0\mathcal{K}\mathcal{C}^{2}}{4(a+K_0\mathcal{K}\mathcal{C}+\delta-\varepsilon)^{2}}e^{2(a+K_0\mathcal{K}\mathcal{C}+\delta-\varepsilon)\sigma}e^{2\nu_1 |t|}.\\
% &=&\displaystyle\frac{\vartheta_1(\sigma)K_0e^{\frac{K_0\mathcal{K}\mathcal{C}}{\gamma-(\varepsilon+\delta)}}\mathcal{K}\mathcal{C}}{4(-a+\varepsilon+\delta)(\gamma-(\varepsilon+\delta))}e^{2\nu_1 t}.
\end{array}$$

In addition, by considering Lemma \ref{RobustezNUBG}, \eqref{Kcondition} and \eqref{conditionGED} ($\delta>\varepsilon$) we have that:
\begin{equation}
\label{PhiK}
\begin{array}{rcl}
&&\displaystyle \int _{t}^{s}\left\| \Phi_{\tilde{A}}(s,p)\right \|^{2}\left\| K(p)\right\|^{2}dp\\
&\leq& \displaystyle\int_{t}^{s}K_0e^{2(a+K_0\mathcal{K}\mathcal{C})(s-p)+2\varepsilon|p|}\mathcal{K}e^{-2\delta |p|}dp,\\
&=& \displaystyle\int_{t}^{s}K_0\mathcal{K}e^{2(a+K_0\mathcal{K}\mathcal{C})(s-p)}e^{-2(\delta-\varepsilon) |p|}dp,\\
&\leq& \displaystyle\int_{t}^{s}K_0\mathcal{K}e^{2(a+K_0\mathcal{K}\mathcal{C}+\delta-\varepsilon)(s-p)}e^{-2(\delta-\varepsilon) |s|}dp,\\
&=&\displaystyle \frac{K_0\mathcal{K}}{2(a+K_0\mathcal{K}\mathcal{C}+\delta-\varepsilon)}e^{-(\delta-\varepsilon)|s|}\left [ -1+e^{2(a+K_0\mathcal{K}\mathcal{C}+\delta-\varepsilon)(s-t)}\right ],\\
&\leq&\displaystyle \frac{K_0\mathcal{K}}{2(a+K_0\mathcal{K}\mathcal{C}+\delta-\varepsilon)}e^{-2(\delta-\varepsilon)|s|}e^{2(a+K_0\mathcal{K}\mathcal{C}+\delta-\varepsilon)(s-t)}.
\end{array}
\end{equation}

Therefore, by mixing Cauchy-Schwarz inequality, \eqref{MTeorema} and \eqref{PhiK}, then we have that:
$$\begin{array}{rcl}
 &&\tilde{\Delta}(t,t+\sigma)\\&\leq&\displaystyle\int_{t}^{t+\sigma}\frac{K_0\mathcal{K}\mathcal{C}^{2}}{2(a+K_0\mathcal{K}\mathcal{C}+\delta-\varepsilon)}e^{-2(\delta-\varepsilon-\gamma)|s|}e^{2(a+K_0\mathcal{K}\mathcal{C}+\delta-\varepsilon)(s-t)}\vartheta_1(\sigma)e^{2\nu_1 |t|}ds,\\
&\leq&\displaystyle\frac{\vartheta_1(\sigma)K_0\mathcal{K}\mathcal{C}^{2}}{4(a+K_0\mathcal{K}\mathcal{C}+\delta-\varepsilon)^{2}}e^{2(a+K_0\mathcal{K}\mathcal{C}+\delta-\varepsilon)\sigma}e^{2\nu_1 |t|}.
 % &\leq&\displaystyle \int_{t}^{t+\sigma}\left \| C(s)\int_{t}^{s}\Phi_{\tilde{A}}(s,p)K(p)C(p)\Phi_{A}(p,t)v\;dp\right \|^{2} ds\\
% &\leq&\displaystyle\int_{t}^{t+\sigma}\mathcal{C}^{2}e^{2\gamma |s|}\left(\int_{t}^{s}\left\|\Phi_{\tilde{A}}(s,p) \right \|^{2}\left \|K(p) \right \|^{2}dp \right)\left (\int_{t}^{s}\left\| C(p)\Phi_{A}(p,t)v\right\|^{2}dp \right)ds,\\
% &\leq&\displaystyle\int_{t}^{t+\sigma}\frac{K_0\mathcal{K}\mathcal{C}^{2}}{2(a+K_0\mathcal{K}\mathcal{C}+\delta-\varepsilon)}e^{-2(\delta-\varepsilon-\gamma)|s|}e^{2(a+K_0\mathcal{K}\mathcal{C}+\delta-\varepsilon)(s-t)}\vartheta_1(\sigma)e^{2\nu_1 |t|}ds,\\
% &\leq&\displaystyle\frac{\vartheta_1(\sigma)K_0\mathcal{K}\mathcal{C}^{2}}{4(a+K_0\mathcal{K}\mathcal{C}+\delta-\varepsilon)^{2}}e^{2(a+K_0\mathcal{K}\mathcal{C}+\delta-\varepsilon)\sigma}e^{2\nu_1 |t|}.\\
% % &=&\displaystyle\frac{\vartheta_1(\sigma)K_0e^{\frac{K_0\mathcal{K}\mathcal{C}}{\gamma-(\varepsilon+\delta)}}\mathcal{K}\mathcal{C}}{4(-a+\varepsilon+\delta)(\gamma-(\varepsilon+\delta))}e^{2\nu_1 t}.
\end{array}$$

On this way, \eqref{igualdadinicialcaso2} can be expressed, by considering the respective bounds of the terms, as follows:
$$\vartheta_0(\sigma)e^{-2\nu_0 |t|}\leq -v^{T}M_{\tilde{A}}(t,t+\sigma)v+2\sqrt{\vartheta_1(\sigma)e^{2\nu_1 |t|}}\sqrt{v^{T}M_{\tilde{A}}(t,t+\sigma)v}+\psi(\sigma) e^{2\nu_1 |t|}$$
where 
$$\psi(\sigma)=\displaystyle\frac{\vartheta_1(\sigma)K_0\mathcal{K}\mathcal{C}^{2}e^{2(a+K_0\mathcal{K}\mathcal{C}+\delta-\varepsilon)\sigma}}{4(a+K_0\mathcal{K}\mathcal{C}+\delta-\varepsilon)^{2}},$$
then we have that
$$\begin{array}{rcl}
     \sqrt{v^{T}M_{\tilde{A}}(t,t+\sigma)v}&\leq&\sqrt{\vartheta_1(\sigma)e^{2\nu_1 |t|}}+\sqrt{\vartheta_1(\sigma)e^{2\nu_1 |t|}+(\psi(\sigma) e^{2\nu_1 |t|}-\vartheta_0(\sigma)e^{-2\nu_0 |t|})},  \\
     \sqrt{v^{T}M_{\tilde{A}}(t,t+\sigma)v}&\leq&2 \sqrt{(\vartheta_1(\sigma)+\psi(\sigma)) e^{2\nu_1 |t|}},\\
     v^{T}M_{\tilde{A}}(t,t+\sigma)v&\leq&4(\vartheta_1(\sigma)+\psi(\sigma)) e^{2\nu_1 |t|}.
\end{array}$$

In summary, the following expressions correspond to the estimates of $v^{T}M_{\tilde{A}}(t,t+\sigma)v$: If $\varphi(\sigma)-1\leq0$ we have that
$$\displaystyle\frac{\vartheta^{2}_0(\sigma)}{4\vartheta_1(\sigma)}e^{-(4\nu_0+2\nu_1)|t|}\leq v^{T}M_{\tilde{A}}(t,t+\sigma)v\leq4(\vartheta_1(\sigma)+\psi(\sigma)) e^{2\nu_1 |t|}$$
and if $\varphi(\sigma)-1>0$ we obtain that:
$$\displaystyle\frac{\vartheta_0^{2}(\sigma)e^{-(4\nu_0+2\nu_1)|t|}}{4\max\{\vartheta_1(\sigma),(\varphi(\sigma)-1)\vartheta_0(\sigma)\}}\leq v^{T}M_{\tilde{A}}(t,t+\sigma)v\leq4(\vartheta_1(\sigma)+\psi(\sigma)) e^{2\nu_1 |t|},$$
which proves the first estimate of non-uniform complete observability of the system \eqref{controlA-KC}. In order to prove the estimates for 
$$N_{\tilde{A}}(t+\sigma,t)=\Phi_{\tilde{A}}(t,t+\sigma)M_{\tilde{A}}(t,t+\sigma)\Phi_{\tilde{A}}(t,t+\sigma),$$ 
we can combine the previous estimation and the non-uniform bounded growth property by using Theorem \ref{p2l3obsNOUNIFORME}, which allows us to conclude the non-uniform complete observability of the system \eqref{controlA-KC}.
\end{proof}

\section{Main result; A reciprocal result about the detectability and observability}

In this section, we will use the result of the previous section in order to complete the equivalence between the concepts of non-uniform complete observability and non-uniform exponential detectability\footnote{In the uniform framework a complete equivalence was proved in \cite[Proposition 1]{Tranninger}.}. Specifically, in \cite[Theorem 9]{HM} it was proven that assuming the NUCO property, we obtain the \textit{non-uniform exponential detectability} (NUED), while here we will prove the implication in the opposite sense. A detailed definition of NUED is provided below.

It is important to emphasize that in \cite[Subsection 6.3]{HM} it was proven with a counterexample that it cannot necessarily be guaranteed that the property of non-uniform exponential detectability implies non-uniform complete observability, but in this section it will be proven that under additional conditions, we can achieve observability in the non-uniform framework. 

We will begin by providing context for the most important aspects of non-uniform exponential detectability. 

\begin{definition}
    The linear system
    \begin{equation}
  \label{ALS}
 \dot{x}=V(t)x
 \end{equation}
    is said to be nonuniformly exponentially stable \textnormal{(NUES)} forward if there exist constants $\beta>0$, $\delta\geq0$  and $M\geq1$ such that for all $s\in\mathbb{R}$ and for all $t\geq s$:
\begin{equation}
\label{NUESInfinito}
\|\Phi_{V}(t,s)\|\leq Me^{\delta |s|}e^{-\beta(t-s)}.
\end{equation}
\end{definition}

\begin{remark}
\label{estabilidadNU}
We emphasize that in Definition 15 ii) in \textnormal{\cite{HMR}}, it is imposed the condition $\delta\in[0,\beta)$, which is framed in the spectral perspective discussed in that paper.
\end{remark}

\begin{definition}
    The linear system \eqref{ALS} is said to be nonuniformly exponentially stable \textnormal{(NUES)} backward if there exist constants $\beta>0$, $\delta>0$  and $M\geq1$ such that for all $s\in\mathbb{R}$ and for all $t\leq s$:
\begin{equation}
\label{NUESMenosInfinito}
\|\Phi_{V}(t,s)\|\leq Me^{\delta |s|}e^{\beta(t-s)}.
\end{equation}
\end{definition}

% From the concept of nonuniform exponential stability, we can extend the notion of uniform exponential detectability to this new framework.

% By considering the system \eqref{control1a}-\eqref{control1b}, the exponential detectability consists of guaranteeing the existence of an observer $L:\mathbb{R}\to M_{n\times p}(\mathbb{R})$ described by the equation
% \begin{equation*}
% % \label{observador}
% \dot{\hat{x}}(t)=A(t)\hat{x}(t)+L(t)[y(t)-C(t)\hat{x}(t)],
% \end{equation*}
% such that the linear system
% \begin{equation}
%         \label{SistError}
%         \dot{e}(t)=[A(t)-L(t)C(t)]e(t)
%     \end{equation}
% will be exponentially stable, with $e(t)=x(t)-\dot{\hat{x}}(t)$ being the estimation error

\begin{definition}
    The system \eqref{control} is called nonuniformly exponentially detectable \textnormal{(NUED)} if there exists an observer $L:\mathbb{R}\to M_{n\times p}(\mathbb{R})$ described by the equation
\begin{equation*}
% \label{observador}
\dot{\hat{x}}(t)=A(t)\hat{x}(t)+L(t)[y(t)-C(t)\hat{x}(t)],
\end{equation*}
such that the linear system 
\begin{equation}
        \label{SistError}
        \dot{e}(t)=\tilde{E}(t)e(t), \quad \textnormal{with}\; \tilde{E}(t)=A(t)-L(t)C(t)
    \end{equation}
is \textnormal{NUES} forward, with $e(t)=x(t)-\dot{\hat{x}}(t)$ being the observer estimation error.
\end{definition}

Based on the previous definitions and results, we state the main theorem of this work. 
\begin{comment}
(In the uniform framework, complete equivalence was proved in \cite[Proposition 1]{Tranninger}).
\end{comment}
\begin{theorem}
\label{reciprocodetectabilidadyobservabilidad}
    Suppose that the system \eqref{lin} admits non-uniform bounded growth with constants $K_0$, $a$ and  $\varepsilon$, and additionally verifies the \textnormal{NUES} backward property with constants $K$, $\alpha$ and $\mu$ satisfying $\alpha>\mu$. On the other hand, if the system \eqref{SistError} is \textnormal{NUES} forward, that is, the system verify the \textnormal{NUED} condition, the estimate \eqref{Ccondition} is satisfied and the following conditions hold:
\begin{equation}
\label{cotaL}
\|L(t)\|\leq \mathcal{L}e^{-\ell |t|},\quad \textnormal{with}\; \ell>\alpha+\mu
\end{equation}
\begin{equation}
\label{CotasDeExponentes}
\ell>\gamma+\varepsilon
\end{equation}
then the system \eqref{control} is non-uniform completely observable.
\end{theorem}

\begin{proof}
    In order to obtain a contradiction, we assume that the system \eqref{control} does not satisfy the non-uniform complete observability property. In this first part, we will describe the hypotheses and their respective consequences.

    % To begin with, we will consider all initial hypotheses, in addition to assuming that system \eqref{SistError} verifies the NUES forward property and that system \eqref{control} does not satisfy complete non-uniform observability, all with the aim of obtaining a contradiction.

    % In this first part, we will describe the hypotheses and their respective consequences.

    Firstly, if the system \eqref{SistError} admits the NUES forward property, then there exist constants $K_e$, $\alpha_e$ and $\mu_e$ such that for all $t_1\in\mathbb{R}$ and for all $t_2\geq t_1$, we have that
    \begin{equation}
    \label{cotaEtildeNUESforward}
\left\|\Phi_{\tilde{E}}\left(t_2, t_1\right)\right\| \leq K_{e} e^{-\alpha_{e}\left(t_2-t_1\right)+\mu_{e}|t_1|},
\end{equation}
If system \eqref{control1a} is NUES backward, then there exist constants $K$, $\alpha$ and $\mu$ such that for all $t_1\in\mathbb{R}$ and for all $t_2\geq t_1$ we have that
\begin{equation}
\label{cotaANUESbackward}
\left\|\Phi_{A}\left(t_1, t_2\right)\right\| \leq K e^{\alpha\left(t_1-t_2\right)+\mu|t_2|}
\end{equation}

On the other hand, we note that if the system \eqref{control1a} has non-uniform bounded growth with constants $K_0$, $a$ and $\varepsilon$, then Lemma \ref{RobustezNUBG} guarantees that the system \eqref{SistError} also verifies the property of non-uniform bounded growth with constants $K_0$, $a+K_0\mathcal{K}\mathcal{C}$ and $\varepsilon$.

Thus, if any of the conditions \eqref{Mobserv}-\eqref{Mobserv2} are satisfied, the other one is guaranteed,  by Theorem \ref{p2l3obsNOUNIFORME}. For this reason, we will deny both conditions \eqref{Mobserv}-\eqref{Mobserv2}, which directly affects system 
\begin{subequations}\label{controlA-LC}
	\begin{empheq}[left=\empheqlbrace]{align}
		& \dot{x}(t)=(A(t)-L(t)C(t))x(t), \label{control1A-LC} \\
		& y(t)=C(t)x(t), \label{control1C}
	\end{empheq}
\end{subequations}
since based on conditions \eqref{Ccondition}, \eqref{cotaL} and \eqref{CotasDeExponentes}, Theorem \ref{observableconK} guarantees that the system \eqref{controlA-LC}\footnote{\eqref{controlA-LC} denotes the control system \eqref{control1A-LC}-\eqref{control1C}} also does not admit non-uniform complete observability.

Therefore, if we deny the non-uniform complete observability of the system \eqref{control1A-LC}-\eqref{control1C}, then we have that:\\

\noindent\textbf{(N)} $\forall\; \nu_0, \nu_1,\tilde{\nu}_0,\tilde{\nu}_1\geq0$,
$\forall\;\vartheta_0(\cdot),\vartheta_1(\cdot),\varrho_0(\cdot), \varrho_1(\cdot)$ non-negative functions, there exists $t\in \mathbb{R}$ such that $\forall \sigma_0>0$, there exists $\sigma\geq\sigma_0$ such that neither of the following two estimates is satisfied:
\begin{equation}
\label{MobservE}
e^{-2\nu_0 |t|}\vartheta_0(\sigma)I\leq M_{\tilde{E}}(t,t+\sigma)\leq  e^{2\nu_1 |t|}\vartheta_1(\sigma)I,
\end{equation}
\begin{equation}
\label{Mobserv2E}
e^{-2\bar{\nu}_0 |t|}\varrho_0(\sigma)I\leq N_{\tilde{E}}(t,t+\sigma)\leq  e^{2\bar{\nu}_1 |t|}\varrho_1(\sigma)I.
\end{equation}

If we focus only on the negation of \eqref{MobservE}, we will have that 
\begin{equation}
\label{NegacionMObserv}
M_{\tilde{E}}(t,t+\sigma)<e^{-2\nu_0 |t|}\vartheta_0(\sigma)I \quad \textnormal{or}\quad e^{2\nu_1 |t|}\vartheta_1(\sigma)I<M_{\tilde{E}}(t,t+\sigma).
\end{equation}

Essentially, we will choose the appropriate $\nu$ and $\vartheta_0$ and, for the co\-rres\-pon\-ding $t$, we will choose a suitable $\sigma_0$, in order to find $\sigma\geq\sigma_0$ that does not fulfill \eqref{MobservE}. This will lead to a contradiction. The transition matrix of \eqref{SistError} is given by:
    $$
\Phi_{\tilde{E}}\left(t_2, t_1\right)=\Phi_{A}\left(t_2, t_1\right)-\int_{t_1}^{t_2} \Phi_{A}\left(t_2, s\right) L(s) C(s) \Phi_{\tilde{E}}\left(s, t_1\right) d s,
$$
for all $t_1, t_2\in\mathbb{R}$. Multiplying by the left side of the previous identity by $\Phi_{A}\left(t_1, t_2\right)$ we have that:

$$
\begin{array}{rcl}
\displaystyle\Phi_{A}\left(t_1, t_2\right) \Phi_{\tilde{E}}\left(t_2, t_1\right)&=&\displaystyle I-\int_{t_1}^{t_2} \Phi_{A}\left(t_1, s\right) L(s) C(s) \Phi_{\tilde{E}}\left(s, t_1\right) d s\\
I&=&\displaystyle\Phi_{A}\left(t_1, t_2\right) \Phi_{\tilde{E}}\left(t_2, t_1\right)+\int_{t_1}^{t_2} \Phi_{A}\left(t_1, s\right) L(s) C(s) \Phi_{\tilde{E}}\left(s, t_1\right) d s.
\end{array}
$$

By considering this identity and applying it to a non trivial vector $v\in\mathbb{R}^{n}$ we have that:

\begin{equation}
\label{estimacionV}
\begin{array}{rcl}
\|v\|&\leq& \|\Phi_{A}\left(t_1, t_2\right)\| \| \Phi_{\tilde{E}}\left(t_2, t_1\right)\|\|v\|\\
&+&\displaystyle\int_{t_1}^{t_2}\|\Phi_{A}\left(t_1, s\right)\| \|L(s)\| \|C(s) \Phi_{\tilde{E}}\left(s, t_1\right)v\| ds
\end{array}
\end{equation}
If we replace \eqref{cotaL}, \eqref{cotaEtildeNUESforward} and \eqref{cotaANUESbackward} in \eqref{estimacionV}, we have that for $t_1\leq t_2$:
\begin{equation}
\label{estimacionVdetalle}
\begin{array}{rcl}
\|v\|
&\leq& \displaystyle Ke^{\alpha(t_1-t_2)+\mu|t_2|}K_{e}e^{-\alpha_{e}(t_2-t_1)+\mu_{e}|t_1|}\|v\|\\
&+&\displaystyle\int_{t_1}^{t_2}Ke^{\alpha(s-t_1)+\mu|s|}\mathcal{L}e^{-\ell|s|}\|C(s)\Phi_{\tilde{E}}(s,t_1)v\| ds\\
&=&\displaystyle KK_{e}e^{-(\alpha+\alpha_{e})(t_2-t_1)+\mu|t_2|+\mu_{e}|t_1|}\|v\|\\
&+&\displaystyle\int_{t_1}^{t_2}K\mathcal{L}e^{\alpha(s-t_1)-(\ell-\mu)|s|}\|C(s)\Phi_{\tilde{E}}(s,t_1)v\| ds
\end{array}
\end{equation}

Based on condition \textbf{(N)} we define $\nu_{0}=\alpha$ and 
$$\vartheta_0=\displaystyle\frac{1}{9(K\mathcal{L}w)^{2}}, \quad \textnormal{with } w=\left(\frac{1}{\sqrt{2(\alpha+\ell-\mu)}}+\frac{1}{\sqrt{2|\alpha-\ell+\mu|}}\right),$$
then there exists $t\in\mathbb{R}$ such that if we choose $\sigma_0>0$ which satisfies the following conditions:  
\begin{equation} \label{Sigma0Abajo}
\mathcal{A}(t)=\displaystyle\frac{\ln(3KK_e)}{\alpha-\mu+\alpha_{e}}+\frac{\mu+\mu_{e}}{\alpha-\mu+\alpha_{e}}|t|\leq \sigma_{0} \quad \textnormal{and}\quad 0<t+\sigma_0,
\end{equation}
there exists $\sigma\geq\sigma_0$ such that the first inequality of \eqref{NegacionMObserv} is verified.

By considering the integral term of \eqref{estimacionVdetalle}, with $t_1=t$ and $t_2=t+\sigma$, we can use Cauchy-Schwarz inequality and the first estimate of \eqref{NegacionMObserv} in order to work the second term of \eqref{estimacionVdetalle}:
$$\begin{array}{rcl}
&&\displaystyle\int_{t}^{t+\sigma}K\mathcal{L}e^{\alpha(s-t)-(\ell-\mu)|s|}\|C(s)\Phi_{\tilde{E}}(s,t)v\| ds\\
&\leq&\left(\displaystyle\int_{t}^{t+\sigma}\left(K\mathcal{L}e^{\alpha(s-t)-(\ell-\mu)|s|}\right)^{2}\right)^{\frac{1}{2}}\cdot\left(\displaystyle\int_{t}^{t+\sigma}\|C(s)\Phi_{\tilde{E}}(s,t)v\|^{2} ds\right)^{\frac{1}{2}}\\
&\leq&\left(\displaystyle\int_{t}^{t+\sigma}K^{2}\mathcal{L}^{2}e^{2\alpha(s-t)-2(\ell-\mu)|s|}ds\right)^{\frac{1}{2}}e^{-\nu_{0}|t|}\sqrt{\vartheta_{0}}\|v\|.
\end{array}$$

We emphasize the fact that condition \textbf{(N)} gives rise to two cases for the possible values of $t$ and, consequently, two ways of estimating the second term of \eqref{estimacionVdetalle}:
\begin{itemize}
    \item [$\bullet$] Case 1: $t<0$. By considering $\sigma\geq\sigma_0$ and the second condition of \eqref{Sigma0Abajo}, then $t+\sigma>0$ and we have that:
$$\left(\displaystyle\int_{t}^{t+\sigma}K^{2}\mathcal{L}^{2}e^{2\alpha(s-t)-2(\ell-\mu)|s|}ds\right)^{\frac{1}{2}}e^{-\nu_{0}|t|}\sqrt{\vartheta_{0}}\|v\| =  I_1+I_2,$$
where
$$I_1=\left(\displaystyle\int_{t}^{0}K^{2}\mathcal{L}^{2}e^{2\alpha(s-t)-2(\ell-\mu)|s|}ds\right)^{\frac{1}{2}}e^{-\nu_{0}|t|}\sqrt{\vartheta_{0}}\|v\|$$
and
$$I_2=\left(\displaystyle\int_{0}^{t+\sigma}K^{2}\mathcal{L}^{2}e^{2\alpha(s-t)-2(\ell-\mu)|s|}ds\right)^{\frac{1}{2}}e^{-\nu_{0}|t|}\sqrt{\vartheta_{0}}\|v\|$$

    For $I_1$, due to estimate \eqref{cotaL}, we observe that $\alpha+\ell-\mu>0$, then we have that: 
    $$\begin{array}{rcl}
I_1&=&\sqrt{K^{2}\mathcal{L}^{2}\vartheta_0}e^{-\alpha t}\displaystyle\left(\int_{t}^{0}e^{2(\alpha+\ell-\mu)s}ds\right)^{\frac{1}{2}}e^{-\nu_{0}|t|}\|v\|\\
    &=&\displaystyle\sqrt{\frac{K^{2}\mathcal{L}^{2}\vartheta_0}{2(\alpha+\ell-\mu)}}e^{\alpha |t|}e^{-\nu_{0}|t|}\left(1-e^{2(\alpha+\ell-\mu)t}\right)^{\frac{1}{2}}\|v\|\\
    &\leq&\displaystyle\sqrt{\frac{K^{2}\mathcal{L}^{2}\vartheta_0}{2(\alpha+\ell-\mu)}}e^{(\alpha-\nu_0) |t|}\|v\|=\displaystyle \sqrt{\frac{K^{2}\mathcal{L}^{2}\vartheta_0}{2(\alpha+\ell-\mu)}}\|v\|
    \end{array}$$

On the other hand, from \eqref{cotaL} we see that $a-\ell+\mu<0$, then for $I_2$ we have that:

$$\begin{array}{rcl}
    I_2&=&\sqrt{K^{2}\mathcal{L}^{2}\vartheta_0}e^{-\alpha t}\displaystyle\left(\int_{0}^{t+\sigma}e^{2(\alpha-\ell+\mu)s}ds\right)^{\frac{1}{2}}e^{-\nu_{0}|t|}\|v\|\\
    &=&\displaystyle\sqrt{\frac{K^{2}\mathcal{L}^{2}\vartheta_0}{2|\alpha-\ell+\mu|}}e^{\alpha |t|}e^{-\nu_{0}|t|}\left(1-e^{2(\alpha-\ell+\mu)(t+\sigma)}\right)^{\frac{1}{2}}\|v\|\\
    &\leq&\displaystyle\sqrt{\frac{K^{2}\mathcal{L}^{2}\vartheta_0}{2|\alpha-\ell+\mu|}}e^{(\alpha-\nu_0) |t|}\|v\|=\displaystyle \sqrt{\frac{K^{2}\mathcal{L}^{2}\vartheta_0}{2|\alpha-\ell+\mu|}}\|v\|
    \end{array}$$

    Thus we can ensure that:
$$\left(\displaystyle\int_{t}^{t+\sigma}K^{2}\mathcal{L}^{2}e^{2\alpha(s-t)-2(\ell-\mu)|s|}ds\right)^{\frac{1}{2}}e^{-\nu_{0}|t|}\sqrt{\vartheta_{0}}\|v\| \leq w\sqrt{K^{2}\mathcal{L}^{2}\vartheta_0} \|v\|= \frac{1}{3}\|v\|$$
    
    \item[$\bullet$] Case 2: $t\geq0$. From \eqref{cotaL}, we know that $\alpha-\ell+\mu<0$, then we have the following estimate:

    $$\begin{array}{rcl}
         &&\left(\displaystyle\int_{t}^{t+\sigma}K^{2}\mathcal{L}^{2}e^{2\alpha(s-t)-2(\ell-\mu)|s|}ds\right)^{\frac{1}{2}}e^{-\nu_{0}|t|}\sqrt{\vartheta_{0}}\|v\|\\ &\leq&  \sqrt{K^{2}\mathcal{L}^{2}\vartheta_0}e^{-\alpha |t|}\displaystyle\left(\int_{t}^{t+\sigma}e^{2(\alpha-\ell+\mu)s}ds\right)^{\frac{1}{2}}e^{-\nu_{0}|t|}\|v\|\\
    &=&\displaystyle\sqrt{\frac{K^{2}\mathcal{L}^{2}\vartheta_0}{2|\alpha-\ell+\mu|}}e^{-\alpha |t|}e^{-\nu_{0}|t|}\left(e^{2(\alpha-\ell+\mu)t}-e^{2(\alpha-\ell+\mu)(t+\sigma)}\right)^{\frac{1}{2}}\|v\|\\
    &\leq&\displaystyle\sqrt{\frac{K^{2}\mathcal{L}^{2}\vartheta_0}{2|\alpha-\ell+\mu|}}e^{-(\alpha+\nu_0) |t|}\|v\|
    \leq\displaystyle \sqrt{\frac{K^{2}\mathcal{L}^{2}\vartheta_0}{2|\alpha-\ell+\mu|}}\|v\|\\
    &\leq& \displaystyle w\sqrt{K^{2}\mathcal{L}^{2}\vartheta_0}\|v\|=\frac{1}{3}\|v\|
    \end{array}$$

\end{itemize}

On the other hand, in relation to the first term of \eqref{estimacionVdetalle}, we can note that the first condition of \eqref{Sigma0Abajo} allows us to state that for $\sigma_0\leq\sigma$, we obtain the following:
$$KK_{e}e^{-(\alpha-\mu+\alpha_{e})\sigma+(\mu+\mu_{e})|t|}\|v\|\leq\frac{1}{3}\|v\|.$$

Based on the previous work, we can estimate \eqref{estimacionVdetalle}:
$$\|v\|\leq \frac{1}{3}\|v\|+\frac{1}{3}\|v\|=\frac{2}{3}\|v\|$$
which is a contradiction.
\end{proof}

\begin{remark}
As appears in \textnormal{\cite [Proposition 1]{Tranninger}}, the previous theorem includes as a hypothesis the \textnormal{NUES} forward condition for the system \eqref{lin} (in \textnormal{\cite{Tranninger}} this type of system is called anti-stable), but additionally considers the hypothesis in relation to the growth rate and the non-uniform parameter, which allows us to make the necessary estimates to achieve the result. This type of hypothesis is quite common, especially when dealing with spectral theory associated with the non-uniform exponential dichotomy (see Remark \ref{estabilidadNU}). A complete and detailed spectral theory was described in \textnormal{\cite{Silva2023, Zhang}}. 
\end{remark}

\section{Consequence of duality; relationship between controllability and stabilizability}

\subsection{Preservation of the non-uniform complete controllability via state feedback}

This subsection will present an important consequence of the proved dua\-li\-ty between non-uniform complete controllability and non-uniform complete observability. Specifically, based on Theorem \ref{observableconK}, the preservation of non-uniform complete controllability via input feedback will be established.

By considering the linear state variable feedback $u(t)=-F(t)x(t)+v(t)$, with $v(t)$ a new input, and replacing this expression in
\eqref{control1a} we obtain the following system:
\begin{equation}
    \label{LTVControlabilidadperturbada}
    \dot{x}(t)=(A(t)-B(t)F(t))x(t)+B(t)v(t).
\end{equation}

The following theorem seeks to prove the same idea as Theorem \ref{observableconK}, that is, to prove the preservation of non-uniform complete controllability via state feedback.

\begin{theorem}
Assume that the system \eqref{lin} admits non-uniform bounded growth with constants $K_0$, $a$ and $\varepsilon$. If the following estimates are satisfied for any $z\in \mathbb{R}$:
\begin{equation}
\label{L}
\|F^{T}(z)\|\leq\mathcal{L}_{c}e^{-\ell_{c} |z|},
\end{equation}
\begin{equation}
\label{B}
\|B^{T}(z)\|\leq\mathcal{B}_{c}e^{\beta_{c} |z|},
\end{equation}
\begin{equation}
\label{parameters}
\ell_{c}>\beta_{c}+\varepsilon,
\end{equation}
with $\ell_{c}>0$ and $\beta_c\geq0$, then the system \eqref{control1a} is non-uniformly completely controllable if and only if the system \eqref{LTVControlabilidadperturbada} is non-uniformly completely controllable.
\end{theorem}

\begin{proof}
    By using Theorem \ref{adjuntoNU}, the system \eqref{control1a} is non-uniformly completely controllable if and only if the system
    % \begin{subequations}
$$\left\{	\begin{array}{rcl}
    
		 \dot{x}(t)&=&-A^{T}(t)x(t), 
        % \label{control1a} 
        \\
		 y(t)&=&-B^{T}(t)x(t), 
        % \label{control1b}
	\end{array}\right.$$
% \end{subequations}
is non-uniformly completely observable. By using this fact, by Theorem \ref{observableconK} and by considering the conditions \eqref{L}, \eqref{B} and \eqref{parameters}, the system
% \begin{subequations*}
	$$\left\{\begin{array}{rcl}
		 \dot{x}(t)&=&(-A^{T}(t)-F^{T}(t)(-B^{T}(t))x(t), 
        % \label{control1a} 
        \\
		 y(t)&=&-B^{T}(t)x(t), %\label{control1b}
	\end{array}\right.$$
% \end{subequations}
is also non-uniformly completely observable. Finally, by using again the Theo\-rem \ref{adjuntoNU}, we can conclude that the system \eqref{LTVControlabilidadperturbada}
% \begin{equation*}
%     % \label{LTVControlabilidad}
%     \dot{x}(t)=(A(t)-B(t)L(t))x(t)+B(t)v(t)
% \end{equation*}
is non-uniformly completely controllable,
%which corresponds to the transformation of the system \eqref{LTVControlabilidad} by considering the input feedback $u(t)=-L(t)x(t)+v(t)$, 
proving the result.
\end{proof}

\subsection{Stabilizability and its consequence on controllability}

As well as throughout this section, we emphasize the importance of adjoint and dual systems for
a linear system, namely, given the system \eqref{ALS}, we define its \textit{adjoint linear system} and its \textit{dual linear system} respectively as
\begin{equation}
    \label{adjuntoV}
    \dot{x}(t)=-V^{T}(t)x(t),
\end{equation} 
\begin{equation}
    \label{dualV}
    \dot{x}(t)=V^{T}(-t)x(t).
\end{equation} 
% denoting by $\Psi_{a}(\cdot,\cdot)$ and $\Psi_{d}(\cdot,\cdot)$ as their respective transition matrices. 
In this context, it is interesting to know what kind of stability properties each of these two systems possesses based on the stability of the system \eqref{ALS}, as stated in the following theorem (see the proof of Theorem 8 in \cite{HM}).
\begin{theorem}
    \label{EstabilidadOriginalAdjunto} 
    The following three statements are equivalent:
    \begin{itemize}
        \item [i)] The original system \eqref{ALS} is \textnormal{NUES} forward.
        \item [ii)] The adjoint system \eqref{adjuntoV} is \textnormal{NUES} backward.
        \item [iii)] The dual system \eqref{dualV} is \textnormal{NUES} forward.
    \end{itemize}
\end{theorem}

Based on the study of stability, the concept of non-uniform exponential stabilizability is defined.

\begin{definition}
    The system \eqref{control1a} is called non-uniformly exponentially stabilizable \textnormal{(NUEStab)} if there exists a linear feedback input $u(t)=-F(t)x(t)$, with $F:\mathbb{R}\to M_{n\times p}(\mathbb{R})$ such that the following closed loop system is \textnormal{NUES} forward: 
\begin{equation}
\label{estabilizabilidad}
\dot{x}(t)=(A(t)-B(t)F(t))x(t).
\end{equation}
\end{definition}

With all this information we can establish the reciprocal result of Theorem 26 in \cite{HMR} (see \cite[Theorem 3]{Ikeda} for the uniform context of stabilizability and controllability): 

\begin{theorem}\label{teo-teorema-algo}
    Suppose that the system \eqref{control1a} admits non-uniform bounded growth with constants $K_0$, $a$ and $\varepsilon$, and additionally verifies the \textnormal{NUES} backward property with constants $K$, $\alpha$ and $\mu$ satisfying $\alpha>\mu$. On the other hand, if the system \eqref{estabilizabilidad} is \textnormal{NUES} forward, that is, the system verifies \textnormal{NUEStab} condition, the estimate \eqref{Ccondition} is satisfied and the following conditions hold:
\begin{equation}
\label{cotaLtraspuesto}
\|F^{T}(t)\|\leq \mathcal{L}_{c}e^{-\ell_{c} |t|},\quad \textnormal{with}\; \ell_{c}>\alpha+\mu
\end{equation}
\begin{equation}
\label{CotasDeExponentestraspuesto}
\ell_{c}>\beta_{c}+\varepsilon
\end{equation}
then the system \eqref{control1a} is non-uniform completely controllable.

\end{theorem}

\begin{proof} The idea of the proof is sketched in Figure 
\ref{fig-MainTheoProof}. 
    Firstly, if the system \eqref{control1a} is NUES backward, the Theorem \ref{EstabilidadOriginalAdjunto} ensures that the dual system
    $$\dot{x}(t)=A^{T}(-t)x(t)$$
    is NUES backward. In the same way, if the system \eqref{estabilizabilidad} verifies the property of NUES forward, then we use the Theorem \ref{EstabilidadOriginalAdjunto} in order to obtain that the system
    $$
   \begin{array}{rcl} 
   \dot{x}(t)&=&(A(-t)-B(-t)F(-t))^{T}x(t),\\&=&(A^{T}(-t)-F^{T}(-t)B^{T}(-t))x(t)
   \end{array}$$
    is NUES forward. After that, by considering conditions \eqref{cotaLtraspuesto} and \eqref{CotasDeExponentestraspuesto}, the Theorem \ref{reciprocodetectabilidadyobservabilidad} allows us to ensure that the system
    $$\left\{	\begin{array}{rcl}
    
		 \dot{x}(t)&=&A^{T}(-t)x(t), 
        % \label{control1a} 
        \\
		 y(t)&=&B^{T}(-t)x(t), 
        % \label{control1b}
	\end{array}\right.$$
    is NUCO and finally, by Theorem \ref{dualNU} we conclude that the system \eqref{control1a} is NUCC.  
    \end{proof}
    
    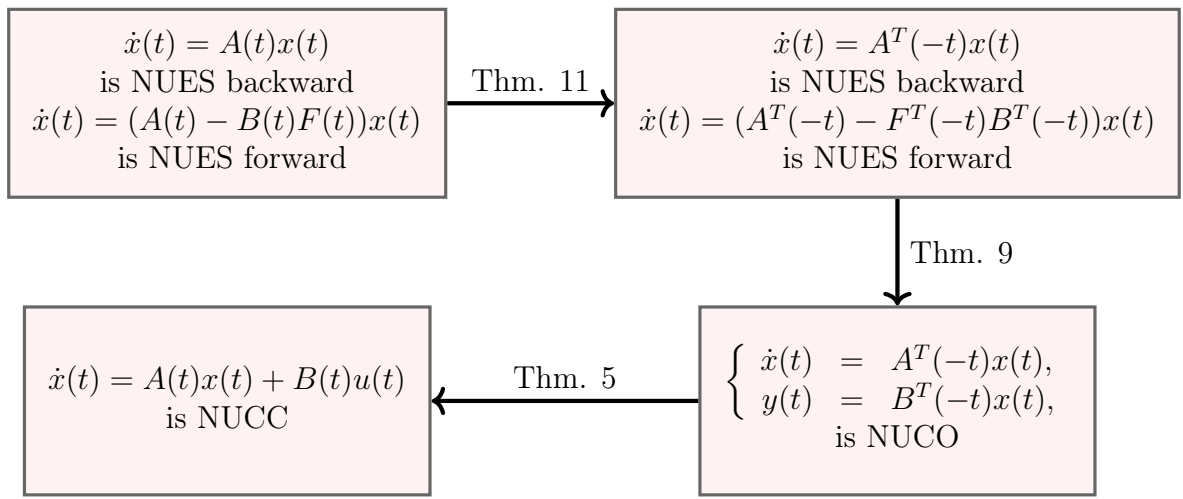
\begin{figure*}[h!]

\centering
\begin{tikzpicture}
[squarednode/.style={rectangle, draw=black!60, fill=red!5, very thick, minimum size=25mm},scale=.99]
%Nodes
\node[squarednode]   at (0,0)   (11)                              {$
\begin{array}{c}
\dot{x}(t)=A(t)x(t)\\
\textnormal{is NUES backward}\\
\dot{x}(t)=(A(t)-B(t)F(t))x(t)\\ \textnormal{ is NUES forward}
\end{array}
$
};

\node[squarednode]   at (9,0)    (12)        {$\begin{array}{c}
\dot{x}(t)=A^{T}(-t)x(t)\\
\textnormal{is NUES backward}\\
\dot{x}(t)=(A^{T}(-t)-F^{T}(-t)B^{T}(-t))x(t)\\
\textnormal{is NUES forward}
\end{array}$
};

\node[squarednode]   at (0,-4)   (21)        {$\begin{array}{c}
\dot{x}(t)= A(t)x(t)+B(t)u(t)\\
\textnormal{is NUCC}   
\end{array}
$
};

\node[squarednode]  at (9,-4)    (22)        {$\begin{array}{c}
\left\{	\begin{array}{rcl}
    
		 \dot{x}(t)&=&A^{T}(-t)x(t), 
        % \label{control1a} 
        \\
		 y(t)&=&B^{T}(-t)x(t), 
        % \label{control1b}
	\end{array}\right.\\
    \textnormal{is NUCO}
\end{array}$
};

%Lines
\draw[->,ultra thick] (11) -- (12)
node[midway, above] {Thm. \ref{EstabilidadOriginalAdjunto}};

\draw[->,ultra thick] (12) -- (22)
node[midway, right] {Thm. \ref{reciprocodetectabilidadyobservabilidad}};

\draw[->,ultra thick] (22) -- (21)
node[midway, above] {Thm. \ref{dualNU}};

\end{tikzpicture}

\caption{Graphical scheme of the proof of Theorem \ref{teo-teorema-algo}.}
\label{fig-MainTheoProof}
\end{figure*}

\newpage 
\section{Acknowledgements}
The first author was partially funded by FONDECYT Regular Grant 1240361.

\bibliographystyle{plain}        % Include this if you use bibtex 
%\bibliography{autosam}
\bibliography{samplebib}
% and a bib file to produce the 
                                 % bibliography (preferred). The
                                 % correct style is generated by
                                 % Elsevier at the time of printing.

%\begin{thebibliography}{99}     % Otherwise use the  
                                 % thebibliography environment.
                                 % Insert the full references here.
                                 % See a recent issue of Automatica 
                                 % for the style.
%  \bibitem[Heritage, 1992]{Heritage:92}
%     (1992) {\it The American Heritage. 
%     Dictionary of the American Language.}
%     Houghton Mifflin Company.
%  \bibitem[Able, 1956]{Abl:56}
%     B.~C.~Able (1956). Nucleic acid content of macroscope. 
%     {\it Nature 2}, 7--9. 
%  \bibitem[Able {\em et al.}, 1954]{AbTaRu:54}   
%     B.~C. Able, R.~A. Tagg, and M.~Rush (1954).
%     Enzyme-catalyzed cellular transanimations.
%     In A.~F.~Round, editor, 
%     {\it Advances in Enzymology Vol. 2} (125--247). 
%     New York, Academic Press.
%  \bibitem[R.~Keohane, 1958]{Keo:58}
%     R.~Keohane (1958).
%     {\it Power and Interdependence: 
%     World Politics in Transition.}
%     Boston, Little, Brown \& Co.
%  \bibitem[Powers, 1985]{Pow:85}
%     T.~Powers (1985).
%     Is there a way out?
%     {\it Harpers, June 1985}, 35--47.

%\end{thebibliography}

%\appendix
%\section{A summary of Latin grammar}    % Each appendix must have a short title.
%\section{Some Latin vocabulary}         % Sections and subsections are supported  
                                        % in the appendices.
\end{document}